\documentclass[reqno,12pt,a4paper]{amsart}
\usepackage{amscd,amsmath,amssymb,verbatim,color,mathtools,bbm}
\usepackage{nicefrac}
\usepackage{todonotes}
\usepackage{hyperref}
\usepackage{url}
\usepackage{enumerate,enumitem}
\usepackage{appendix}
\usepackage{epsfig,subfigure}
\usepackage{graphicx}

\theoremstyle{plain}
\newtheorem{theorem}{Theorem}[section]
\newtheorem{proposition}[theorem]{Proposition}
\newtheorem{lemma}[theorem]{Lemma}

\newtheorem{assumption}[theorem]{Assumption}

\theoremstyle{definition}

\newtheorem{remark}[theorem]{Remark}

\usepackage{geometry}

\renewcommand{\Re}{\mathop{\rm Re}\nolimits}
\renewcommand{\Im}{\mathop{\rm Im}\nolimits}

\newcommand{\linspan}{\mathop{\rm span}\nolimits}

\newcommand{\esssup}{\mathop{\rm ess\ sup}}

\newcommand{\supp}{\mathop{\rm supp}\nolimits}
\newcommand{\diver}{\mathop{\rm div}\nolimits}
\newcommand{\curl}{\mathop{\rm curl}\nolimits}

\newcommand{\sign}{\mathop{\rm sign}\nolimits}

\newcommand{\rest}{\left.\kern-2\nulldelimiterspace\right|_}
\newcommand{\norm}[2]{|#1|_{#2}}

\newcommand{\ex}{\mathrm{e}}
\newcommand{\p}{\partial}

\newcommand{\ed}{\mathrm d}
\newcommand{\im}{\mathbf i}

\newcommand{\N}{{\mathbb N}}
\newcommand{\Z}{{\mathbb Z}}
\newcommand{\R}{{\mathbb R}}
\newcommand{\Sp}{{\mathbb S}}
\newcommand{\T}{{\mathbb T}}

\newcommand{\D}{{\mathrm D}}

\newcommand{\BBB}{{\mathbf B}}
\newcommand{\CCC}{{\mathbf C}}

\newcommand{\nnn}{\mathbf n}

\newcommand{\CC}{{\mathcal C}}

\newcommand{\FF}{{\mathcal F}}
\newcommand{\GG}{{\mathcal G}}
\newcommand{\HH}{{\mathcal H}}

\newcommand{\sS}{{\mathcal S}}

\newcommand{\WW}{{\mathcal W}}

\begin{document}

\title{Gevrey regularity for Navier--Stokes equations under Lions boundary conditions}
\author{Duy Phan}\author{S\'ergio S.~Rodrigues}
\address{Johann Radon Institute for Computational and Applied Mathematics,
\"OAW,\newline\indent Altenbergerstra{\normalfont\ss}e 69, A-4040 Linz.\newline \indent
e-mails: {\tt duy.phan-duc@oeaw.ac.at,sergio.rodrigues@oeaw.ac.at}}
\thanks{The authors acknowledge support from the Austrian Science Fund~(FWF): P~26034-N25.}

\begin{abstract}
The Navier--Stokes system is considered in a compact Riemannian manifold. Gevrey class regularity is proven under Lions boundary conditions: in 2D for the 
Rectangle, Cylinder, and  Hemisphere, and in 3D for the Rectangle. The cases of the 2D Sphere and 2D and 3D Torus are also revisited.
\end{abstract}

\maketitle

{\tiny
\tableofcontents
}

\pagestyle{myheadings} \thispagestyle{plain} \markboth{\sc D.Phan and S. S.
Rodrigues}{\sc Gevrey class regularity for Navier--Stokes equations}

\section{Introduction}
Let $\Omega\subset\R^d$, $d\in\{2,\,3\}$ be a connected bounded domain located locally on
one side of its smooth boundary $\Gamma=\p\Omega$. The
Navier--Stokes system, in~$(0,\,T)\times\Omega$, reads
\begin{equation}\label{sys-u-flat}
 \p_t u+\langle u\cdot\nabla\rangle u-\nu\Delta u+\nabla p+h=0, \quad \diver u
 =0,\quad \GG u\rest \Gamma =0,\quad u(0,\,x)=u_0(x)
\end{equation}
where as usual $u=(u_1,\,\dots,\,u_d)$ and~$p$,
defined for $( t,\,x_1,\,\dots,\,x_d)\in I\times\Omega$, are respectively the unknown velocity field and
pressure of the fluid, $\nu>0$ is the viscosity, the operators~$\nabla$ and~$\Delta$ are respectively the
well known gradient and Laplacian in the space variables $(x_1,\,\dots,\,x_d)$,
$\langle u\cdot\nabla\rangle v$ stands for $(u\cdot\nabla v_1,\,\dots,\,u\cdot\nabla v_d)$,
$\diver u\coloneqq \sum_{i=1}^d \p_{x_i}u_i$ and $h$ is a fixed function. Further, $\GG$ is an appropriate linear operator
imposing the boundary conditions.

In the case $\Omega$ is a compact Riemannian manifold, either with or without boundary, the Navier--Stokes equation reads
\begin{equation}\label{sys-u-riem}
 \p_t u+\nabla_u^1u+\nu\Delta_\Omega u+\nabla_\Omega p+h=0, \quad \diver u
 =0,\quad \GG u\rest \Gamma =0,\quad u(0,\,x)=u_0(x).
\end{equation}
That is we just replace the Laplace operator by the Laplace-de~Rham operator,
the gradient operator by the Riemannian gradient operator, and the nonlinear term by the Levy-Civita connection. Recall
that a flat (Euclidean) domain $\Omega\subset\R^d$ can be seen a Riemannian manifold and we have $-\Delta=\Delta_\Omega$,
$\nabla=\nabla_\Omega$ and
$\langle u\cdot\nabla\rangle v=\nabla_u^1u$ (see, e.g., \cite[Chapter~5]{Rod-Thesis08}). That is,~\eqref{sys-u-riem} reads~\eqref{sys-u-flat}
in the Euclidean case. We should say that some authors consider the Navier--Stokes equation
on a Riemannian manifold with a slightly different Laplacian operator and sometimes with on more
term involving the (Ricci) curvature of the Riemannian manifold. In that case, we also recover~\eqref{sys-u-flat}
in the Euclidean case because
the curvature vanishes. Writing the Navier--Stokes as~\eqref{sys-u-riem}, we are
following~\cite{Ilyin91,Ilyin94,CaoRammahaTiti99,FenglerFreeden05,Rod-Thesis08,Rod-wmctf07}; for other writings
we refer to~\cite{Priebe94,CoronFur96}.

Often system~\eqref{sys-u-riem}
can be rewritten as an evolutionary system
\begin{equation}\label{sys-u-evol}
 \dot u+B(u,\,u)+A u+h=0,\quad u(0,\,x)=u_0(x),
\end{equation}
where $\Pi$ is a suitable projection onto a subspace~$H$ of divergence free vector
fields~(cf.~\cite[Chapter~II, Section~3]{FoiManRosaTem01},
\cite[Section~4]{Rod06}, \cite[Section~5.5]{Rod-Thesis08}); formally
$B(u,\,v)\coloneqq\Pi\nabla_u^1v$ and $Au=\nu\Pi\Delta_\Omega u$. Usually~$\Pi\nabla=0$, and we suppose that $h=\Pi h$
(otherwise we have just to take~$\Pi h$ in~\eqref{sys-u-evol} instead). 

The aim of this work is to give some sufficient conditions to guarantee that the solution of system~\eqref{sys-u-riem} lives
in a Gevrey regularity space.

For the case of periodic boundary conditions, that is, for the case $\Omega=\T^d$, the
Gevrey regularity has been proven in the pioneering work~\cite{FoiasTemam89} for the Gevrey
class~$\D(A^\frac{1}{2}\ex^{\psi(t)A^{\frac{1}{2}}})$, provided~$u_0\in\D(A^{\frac{1}{2}})$. These results have been extended
to other Gevrey classes in~\cite{Liu92}, namely~$\D(A^{s}\ex^{\psi(t)A^{\frac{1}{2}}})$, provided~$u_0\in\D(A^{s})$,
with $s>\frac{d}{4}$.
The first observation is that there is a gap, for the value of~$s$, for $d=3$. As far as we know this gap is still open until now.
Here we fill the gap, that is, for $\Omega=\T^3$ the Gevrey regularity holds in~$\D(A^{s}\ex^{\psi(t)A^{\frac{1}{2}}})$,
provided~$u_0\in\D(A^{s})$, with
$s\ge\frac{1}{2}$. Further for $\Omega=\T^2$, it will follow that the Gevrey regularity holds
in~$\D(A^{s}\ex^{\psi(t)A^{\frac{1}{2}}})$, provided~$u_0\in\D(A^{s})$, with
$s>0$.

In the case of the Navier--Stokes in the {2D} Sphere~$\Sp^2$, from our conditions, we can
recover the results obtained in~\cite{CaoRammahaTiti99}, that is to say that
the Gevrey regularity holds in~$\D(A^{s}\ex^{\psi(t)A^{\frac{1}{2}}})$, provided~$u_0\in\D(A^{s})$, with
$s>\frac{1}{2}$.

In the above mentioned cases the manifolds $\T^d$ and $\Sp^2$ are boundaryless, which means that essentially we have no boundary conditions.
Here we consider the case of manifolds with boundary and three new results are obtained under Lions boundary conditions,
namely, in the cases $\Omega$ is either a 2D Rectangle~$(0,\,a)\times(0,\,b)$ or a 2D
Cylinder~$(0,\,a)\times b\Sp^1$, or a 2D Hemisphere~$\Sp^2_+$.
By Lions boundary conditions, in two dimensions, we mean
the vanishing both of the normal component~$u\cdot\nnn$ and of the vorticity~$\nabla^\bot\cdot u$ of the vector field $u$ at the boundary;
the reason of the terminology (also adopted in~\cite{Kelliher06,Rod06}) is the
work done in~\cite[Section~6.9]{Lions69}. However the terminology is not followed by all authors,
for example, in~\cite[Section~3]{IlyinTiti06} they are just called ``stress-free boundary conditions''.
Notice that Lions boundary conditions can be seen as a particular case of (generalized) Navier boundary conditions
(cf.~\cite[Section~1 and Corollary~4.2]{Kelliher06}, cf.~\cite[system~(4.1)-(4.2) and Remark~4.4.1]{Rod-Thesis08}).
The Navier boundary conditions are also defined in three dimensions, and
the particular case considered in~\cite[Equation~(1.4)]{XiaoXin07} would correspond to the three dimensional Lions boundary conditions.
The study of Navier boundary conditions have been addressed by many authors in the last years, either because in some situations they
may be more realistic than no-slip boundary conditions or because they are more appropriate in finding a solution for the Euler system as
a limit of solutions for the Navier--Stoles system as~$\nu$ goes to zero (cf.~\cite{XiaoXin07,WangXinZang12}, \cite[Section~8]{Kelliher06}),
or even the possibility to recover the solution under no-slip boundary conditions
as a limit of solutions under Navier boundary conditions (cf.~\cite{JagerMikelic01}, and conversely (cf.~\cite[Section~9]{Kelliher06}). We
refer also to~\cite{IftimiePlanas06,FeireislNecasova11,ChemetovCiprianoGavrilyuk10,AmroucheSeloula11} and references therein.

In both cases of the Rectangle or Cylinder, we obtain the Gevrey regularity
in~$\D(A^{s}\ex^{\psi(t)A^{\frac{1}{2}}})$, provided~$u_0\in\D(A^{s})$, with
$s>0$. In the case of the Hemisphere we find that the Gevrey regularity
holds in~$\D(A^{s}\ex^{\psi(t)A^{\frac{1}{2}}})$, provided~$u_0\in\D(A^{s})$, with
$s>\frac{1}{2}$.

\medskip
The rest of the paper is organized as follows. In Section~\ref{S:prelim}, we give
the necessary conditions (as assumptions) for the existence
of solutions living in a Gevrey class regularity space. In Section~\ref{S:gevrey},
the Gevrey class regularity is proven under the conditions on
the sequence of nonrepeated eigenvalues of the Stokes operator.
In Section~\ref{S:repeated}, we give the corresponding conditions on the sequence of repeated eigenvalues.
In Section~\ref{S:examples}, we revisit
the cases where $\Omega$ is the Torus~$\T^d$ and the Sphere~$\Sp^2$
and give some new examples, namely the cases of 2D Hemisphere, 2D Rectangle, and 2D Cylinder under Lions boundary conditions. 
Finally, the Appendix gathers some auxiliary results used in the main text.

\medskip\noindent
{\bf Notation.}
\addcontentsline{toc}{subsection}{\qquad\; Notation}
We write~$\R$ and~$\mathbb N$ for the sets of real numbers and nonnegative
integers, respectively, and we define $\R_0\coloneqq(0,\,+\infty)$, and $\mathbb
N_0\coloneqq\mathbb N\setminus\{0\}$. We denote by $\Omega\subset\R$ as a
bounded interval.


Given a Banach space~$X$ and an open subset $O\subset\R^n$, let us denote by~$L^p(O,\,X)$, with either $p\in[1,\,+\infty)$ or $p=\infty$,
the Bochner space
of measurable functions $f:O\to X$, and such that
$|f|_X^p$ is integrable over~$O$, for $p\in[1,\,+\infty)$, and such that $\esssup_{x\in O}|f(x)|_X<+\infty$, for $p=\infty$.
In the case $X=\R$ we recover the usual Lebesgue spaces.
By~$W^{s,p}(O,\,\R)$, for
$s\in\R$, denote the usual Sobolev space of order~$s$. In the case
$p=2$, as usual, we denote $H^s(O,\,\R)\coloneqq W^{s,2}(O,\,\R)$. Recall that
$H^0(O,\,\R)=L^2(O,\,\R)$. For each $s>0$, we recall also that~$H^{-s}(O,\,\R)$ stands for the dual space of
$H_0^s(O,\,\R)=\mbox{ closure of }\{f\in C^\infty(O,\,\R)\mid
\supp f\subset O\}$ in~$H^s(O,\,\R)$. Notice that
$H^{-s}(O,\,\R)$ is a space of distributions.

For a normed space~$X$, we denote by~$|\cdot|_X$ the corresponding
norm; in the particular case $X=\R$ we denote $|\cdot|\coloneqq|\cdot|_\R$. By~$X'$ we denote the dual of~$X$,
and by~$\langle \cdot,\cdot\rangle_{X',X}$ the duality between~$X'$
and~$X$. The dual space is endowed with the usual dual norm:
$|f|_{X'}\coloneqq\sup\{\langle f,x\rangle_{X',X}\mid x\in
X\mbox{ and }|x|_X=1\}$. In the case that~$X$ is a Hilbert space we denote the inner product by
$(\cdot,\cdot)_X$.

%


Given a Riemannian manifold $\Omega=(\Omega,\,g)$ with Riemannian metric tensor $g$, we denote by $T\Omega$ the tangent
bundle of $\Omega$ and by $\ed\Omega$ the volume element of~$\Omega$. We denote by~$H^s(\Omega,\,\R)$ and~$H^s(\Omega,\,T\Omega)$
respectively the Sobolev spaces of functions and vector fields defined in $\Omega$. Recall that if $\Omega=O\subset\R^n$, then
$H^s(O,\,TO)=H^s(O,\,\R^n)\sim (H^s(O,\,\R))^n$.

%

$C,\,C_i$, $i=1,2,\dots$, stand for unessential positive constants.

\section{Preliminaries}\label{S:prelim}
\subsection{The evolutionary Navier--Stokes system}\label{sS:evol_sys}
Given a $d$-dimensional compact Riemannian manifold $\Omega=(\Omega,\,g)$, $d\in\{2,\,3\}$, we (suppose we can)
write the Navier--Stokes system as
an evolutionary system in a suitable
closed subspace~$H\subseteq\{u\in L^2(\Omega,\,T\Omega)\mid\diver u=0\}$ of divergence free vector fields
\begin{equation}\label{sys_u}
 \dot u+\nu A u+B(u)+h=0,\quad u(0)=u_0,
\end{equation}
where $A\coloneqq\Pi\Delta_\Omega$ is the Stokes operator and $B(u)\coloneqq B(u,\,u)$ with
$B(u,\,v)\coloneqq\Pi\nabla_u^1v$ as a bilinear operator.

Here $\Pi$ stands for the orthogonal
projection in $L^2(\Omega,\,T\Omega)$ onto $H$,  $\Delta_\Omega$ stands for the
Laplace--de~Rham operator and $(u,\,v)\mapsto\nabla_{u}^1v$ stands for the Levi--Civita connection (cf.\cite[Chapter~3, Section~3.3]{Jost05}).

Recall that, for a domain $\Omega\in\R^d$, we can identify $T\Omega$ with $\R^d$, 
$\Delta_\Omega=-\Delta$ coincides with the usual Laplacian up to the minus sign, and
$\nabla_u^1v=\langle u\cdot\nabla\rangle v$ (see~\cite[Chapter~5, Sections~5.1 and~5.2]{Rod-Thesis08}, \cite[Section~1]{Ilyin91})

We consider $H$, endowed with the norm inherited from $L^2(\Omega,\,T\Omega)$, as a pivot space,
that is, $H=H'$. Let $V\subseteq H$ be another Hilbert space, such that $A$ maps $V$
onto $V'$. The domain of $A$, in $H$,
is denoted $\D(A)\coloneqq\{u\in H\mid Au\in H\}$.

The spaces $H$, $V$, and $\D(A)$ will depend on the boundary conditions where the fluid will be subjected to.
We assume that the inclusion $V\subseteq H$ is dense, continuous, and compact. In this case, the eigenvalues of $A$,
repeated accordingly with their multiplicity, form an increasing sequence $(\underline\lambda_k)_{k\in\N_0}$,
\[
 0<\underline\lambda_1\le\underline\lambda_2\le\underline\lambda_3\le\underline\lambda_4\le\dots,
\]
with $\underline\lambda_k$ going to $+\infty$ with $k$. 

Consider also the strictly increasing subsequence $(\lambda_k)_k\in\N_0$ of the distinct (i.e. nonrepeated) eigenvalues
\[
 0<\lambda_1<\lambda_2<\lambda_3<\dots;
\]
and denote by $P_k$ the orthogonal projection in $H$ onto the eigenspace $P_kH=\{z  \in H\mid A z =\lambda_k z \}$,
associated with the eigenvalue $\lambda_k$,
\begin{equation}\label{projPk}
P_k\colon H\to P_kH,\quad v\mapsto P_kv; 
\end{equation}
with $v=P_kv+w$ and $(w,\,z)_H=0$ for all $z\in P_kH$.

We define also the trilinear form
\[
b(u,\,v,\,w)\coloneqq \int_\Omega g(\nabla^1_u v,\,w)\,\ed\Omega,
\]
provided the integral is finite,
where $g(\cdot,\,\cdot)$ stands for the scalar product in $T\Omega$ induced by the metric tensor~$g$.

Throughout the paper, we consider the following assumptions:

\begin{assumption}\label{A:A}\ \\
$\bullet$ $V\subset H^1(\Omega,\,T\Omega)$, and
$|u|_V\coloneqq \left(\langle Au,\,u\rangle_{V',\,V}\right)^{\frac{1}{2}}$ defines a norm
equivalent to the one inherited from~$H^1(\Omega,\,T\Omega)$;\\
$\bullet$ $\D(A)\subset H^2(\Omega,\,T\Omega)$, and $|u|_{\D(A)}\coloneqq |Au|_H$ defines a norm
equivalent to the one inherited from~$H^2(\Omega,\,T\Omega)$.
\end{assumption}

\begin{assumption}\label{A:B}
The following properties hold for the trilinear form.\\
$\bullet$ $b(u,\,u,\,v)=0$ if $u\in P_kH$ for some $k\in\N_0$;\\
$\bullet$ $b(u,\,v,\,w)=-b(u,\,w,\,v)$;\\
$\bullet$ $|b(u,\,v,\,w)|\le C|u|_{L^\infty(\Omega,\,T\Omega)}|v|_{H^1(\Omega,\,T\Omega)}|w|_{L^2(\Omega,\,T\Omega)}$;\\
$\bullet$ $|b(u,\,v,\,w)|\le C|u|_{L^2(\Omega,\,T\Omega)}|v|_{H^1(\Omega,\,T\Omega)}|w|_{L^\infty(\Omega,\,T\Omega)}$;\\
$\bullet$ $|b(u,\,v,\,w)|\le C|u|_{L^4(\Omega,\,T\Omega)}|v|_{H^1(\Omega,\,T\Omega)}|w|_{L^4(\Omega,\,T\Omega)}$.
\end{assumption}

\begin{assumption}\label{A:Bnml}
There are real numbers $\beta\ge0$ and $\alpha\in(0,\,1)$ such that, for all triples $(n,\,m,\,l)\in\N_0^3$,
\[
\left\{\begin{array}{l}
        (u,\,v,\,w)\in P_nH\times P_mH\times P_lH,\\
        (B(u+v),\,w)_H\ne 0,
       \end{array}\right.
\quad\mbox{implies}\quad
\lambda_l^\alpha\le \lambda_n^\alpha+\lambda_m^\alpha+\beta.
\]
\end{assumption}

Next, for given $(n,\,m,\,l)\in\N_0^3$, we define the sets
\begin{align*}
 \FF_{n,\,m}^\bullet&\coloneqq\left\{k\in\N_0\,\left|\;\begin{array}{l}(B(u+v),\,w)_H\ne 0,\\ \mbox{ for some }
       (u,\,v,\,w)\in P_nH\times P_mH\times P_kH
       \mbox{ with } n<m 
       \end{array}\right.\right\};\\
 \FF_{n,\,\bullet}^l&\coloneqq\left\{k\in\N_0\,\left|\;\begin{array}{l}(B(u+v),\,w)_H\ne 0,\\ \mbox{ for some }
      (u,\,v,\,w)\in P_nH\times P_kH\times P_lH
       \mbox{ with } n<k 
       \end{array}\right.\right\}.
\end{align*}

\begin{assumption}\label{A:Fnml}
There are $C_\FF\in\N_0$ and $ \zeta \in[0,\,+\infty)$ such that, for all $n\in\N_0$
\[
\sup_{(m,\,l)\in\N_0^2}\left\{ {\rm card}(\FF_{n,\,m}^\bullet),\,{\rm card}(\FF_{n,\,\bullet}^l) \right\}\le C_\FF\lambda_n^{ \zeta },
\]
where ${\rm card}(S)$ stands for the cardinality (i.e., the number of elements) of the set $S$.
\end{assumption}

\begin{remark}
Assumptions~\ref{A:A} and~\ref{A:B} are satisfied in well known settings. In contrast, assumptions~\ref{A:Bnml} and~\ref{A:Fnml} will be satisfied more
seldom and play a
key role to derive the Gevrey class regularity for the solutions of the Navier--Stokes system~\eqref{sys_u}.
\end{remark}

\subsection{Some auxiliary results} We present now some results that will be useful hereafter.

\begin{proposition} \label{P:ab_s}
 For given nonnegative real numbers $a$, $b$, and $s$, with $a+b>0$ and $s>0$, it holds
\[
\begin{array}{ll}
2^{s-1}(a^s+b^s)\leq(a+b)^s\leq a^s+b^s,&\mbox{ for }0< s\leq1;\\
a^s+b^s\leq(a+b)^s\leq 2^{s-1}(a^s+b^s),&\mbox{ for }s\geq1.
\end{array}
\]
\end{proposition}

The proof is given in the Appendix, Section~\ref{sApx:proofP:ab_s}.

\begin{remark}
The constants in Proposition~\ref{P:ab_s} are sharp, in the sense that
\begin{itemize}
 \item for $a=b$, we have
$2^{s-1}(a^s+b^s)=(a+b)^s$ for $s>0$,
\item for either $a=0$ or $b=0$, we have
$(a+b)^s= a^s+b^s$ for $s>0$.
\end{itemize}
\end{remark}

\begin{lemma} \label{L:lambda_s}
Assumption~\ref{A:Bnml} holds only if for all $s>0$ there exists a nonnegative real number $C_{(s,\,\alpha,\,\beta)}>0$
depending only on $(s,\,\alpha,\,\beta,\,\lambda_1)$
such that 
\[
\left\{\begin{array}{l}
        (u,\,v,\,w)\in P_nH\times P_mH\times P_lH,\\
        (B(u+v),\,w)_H\ne 0,
       \end{array}\right.
\quad\mbox{implies}\quad
\lambda_l^s\le C_{(s,\,\alpha,\,\beta,\,\lambda_1)}(\lambda_n^s+\lambda_m^s).
\]
\end{lemma}

\begin{proof}
From Assumption~\ref{A:Bnml}, since $\left( \lambda_k \right)_{k \in \N_0}$ is an increasing sequence, we have that 
\[
\lambda_l^\alpha \le \lambda_n^\alpha + \lambda_m^\alpha + \beta\frac{\lambda_n^\alpha + \lambda_m^\alpha}{2\lambda_1^\alpha}
= \left( 1+ \frac{\beta}{2 \lambda_1^\alpha} \right)  \left( \lambda_n^\alpha + \lambda_m^\alpha \right). 
\]
Now for any $s>0$, it follows that
\[
\lambda_l^s \le \left( 1+ \frac{\beta}{2 \lambda_1^\alpha} \right)^{\frac{s}{\alpha}} D_{\frac{s}{\alpha}}
\left( \lambda_n^s + \lambda_m^s \right)
\]
where the constant $D_{\frac{s}{\alpha}}$ depending only on $\frac{s}{\alpha}$ is given by Proposition~\ref{P:ab_s}.
\end{proof}

\section{Gevrey class regularity}\label{S:gevrey}
Here we show that, under Assumptions~\ref{A:Bnml} and~\ref{A:Fnml} and for suitable data $(u_0,\,h)$,
the solution $u$ of system~\eqref{sys_u}\label{sS:gevrey}
takes its values $u(t)$ in a Gevrey class regularity space. We follow the arguments in~\cite{FoiasTemam89, Liu92,CaoRammahaTiti99}.
\subsection{Gevrey spaces and main theorem}
Let us set a complete orthonormal system $\{W_k\mid k\in\N_0\}$ of
eigenfunctions of the Stokes operator $A$. That is,
\[
 AW_k=\underline\lambda_kW_k,\mbox{ for all }k\in\N_0.
\]
We recall that any given $u\in H$ can be written in a unique way as $u=\sum_{k\in\N_0}u_kW_k$, with $u_k=(u,\,W_k)_H\in\R$.
Now, given $s\ge0$ we may define the power $A^s$ of the Stokes operator as
\[
A^s u\coloneqq\sum_{k\in\N_0}\underline\lambda_k^s u_kW_k,
\]
and we denote its domain by $\D\left( A^s\right)\coloneqq\{u\in H\mid A^su\in H\}$.

Analogously we may define the negative powers $A^{-s}$ as
\[
A^{-s} u\coloneqq\sum_{n\in\N_0}\underline\lambda_k^{-s} u_kW_k,
\]
and $\D\left( A^{-s}\right)\coloneqq\{u\mid A^{-s}u\in H\}$, more precisely $\D\left( A^{-s}\right)$ is the closure of $H$ in the norm
$|u|_{\D\left( A^{-s}\right)}\coloneqq\left(\sum_{k\in\N_0}\underline\lambda_k^{-2s}u_k^2\right)^{\frac{1}{2}}$.

We recall that for $s=\frac{1}{2}$ we have $\D(A^{\frac{1}{2}})=V$. For a more complete discussion
on the fractional powers of a compact operator we refer to~\cite[Chapter~II, Section~2.1]{Temam97}.

Given two more nonnegative real numbers $\sigma$ and $\alpha$, we define the Gevrey operator
\[
A^s e^{\sigma A^{\alpha}}u\coloneqq
\sum_{k\in\N_0}e^{\sigma \underline\lambda_k^{\alpha}}\underline\lambda_k^s u_kW_k,
\]
which domain is the Gevrey space $\D\left( A^s e^{\sigma A^{\alpha}}\right)
\coloneqq\left\{u\in H\mid A^s e^{\sigma A^{\alpha}}u\in H\right\}$.

Notice that, for given $s\ge0$, $\sigma\ge0$, and $\alpha\ge0$ the functions in~$\{W_k\mid k\in\N_0\}$ are also eigenfunctions
for $A^s$ and for $A^s e^{\sigma A^{\alpha}}$. Indeed for any $k\in\N_0$ it follows that
\[
A^sW_k=\underline\lambda_k^s W_k\quad\mbox{and}\quad
A^s e^{\sigma A^{\alpha}}W_k=e^{\sigma \underline\lambda_k^{\alpha}}
\underline\lambda_k^s W_k.
\]
Furthermore the operators $A^s$ and $A^s e^{\sigma A^{\alpha}}$ are selfadjoint; indeed
\[
\begin{array}{rcccl}
 (A^su,\,v)_H&=&\sum\limits_{k\in\N_0}\underline\lambda_k^s u_kv_k&=&(u,\,A^sv)_H,\\
 (A^s e^{\sigma A^{\alpha}}u,\,v)_H&=&\sum\limits_{k\in\N_0}e^{\sigma \underline\lambda_k^{\alpha}}\underline\lambda_k^s u_kv_k
 &=&(u,\,A^s e^{\sigma A^{\alpha}}v)_H.
\end{array}
\]

\begin{theorem} \label{T:main}
 Suppose that the Assumptions~\ref{A:A}, \ref{A:B}, \ref{A:Bnml} and~\ref{A:Fnml} hold, and let the strictly increasing
 sequence of (nonrepeated) eigenvalues $(\lambda_k)_{k\in\N_0}$ of the Stokes operator~$A$ satisfy, for some positive
 real numbers $\rho$ and $\xi$, the relation
 \begin{equation} \label{lambda_k}
  \lambda_k>\rho k^\xi,\quad\mbox{for all }k\in\N_0.
 \end{equation}
Further, let us be given  $\alpha\in(0,\,1)$ as in Assumption~\ref{A:Bnml}, $C_\FF$ and $ \zeta \ge 0$
as in Assumption~\ref{A:Fnml}, $\sigma>0$,
$s>\frac{d+2(\xi^{-1}+  2\zeta  -1)}{4}$, 
$h\in L^{\infty}(\R_0,\,\D(A^{s-\frac{1}{2}}e^{\sigma A^{\alpha}}))$, and $u_0\in\D(A^s)$.\\
Then, there are $T^*>0$ and a unique solution
\begin{equation}\label{reg_u}
 u \in L^{\infty} \left(  (0,T^*), D  \left( A^s e^{\sigma A^{\alpha}} \right) \right) 
\cap  L^{2} \left(  (0,T^*), D  \left( A^{s+\frac{1}{2}} e^{\sigma A^{\alpha}} \right) \right),
\end{equation}
for the Navier Stokes system \eqref{sys_u}.\\
Further, $T^*$ depends on
the data
$\left(|h|_{L^{\infty} (\R_0,\, D \left( A^{s- \frac{1}{2}}e^{\sigma A^{\alpha}} \right)} ,\, \left| A^s u_0  \right|_H  \right)$
and also on the constants $\nu$, $\lambda_1$, $d$, $s$, $\sigma$, $\alpha$, $\beta$, $C_\FF$, $ \zeta $, $\rho$, and $\xi$.
\end{theorem}

The proof is given below, in Section~\ref{sS:proofT:main}.

\subsection{Some preliminary results}
We derive some preliminary results that we will need in the proof of Theorem~\ref{T:main}.
Let $u$ solve system~\eqref{sys_u} and let $\sigma> 0$, $\alpha\in(0,\,1)$ and~$ \zeta $ be real numbers as in
Theorem~\ref{T:main}, and set $\varphi(t) \coloneqq \min (\sigma, t)$.
Following the Remark in~\cite[Section~2.3(iii)]{FoiasTemam89}, we can see that
the function $u^*(t) = e^{\varphi(t) A^{\alpha}} u(t)$
satisfies
$\partial_t u^* = \frac{d \varphi}{dt} A^{\alpha} e^{\varphi A^{\alpha}} u
+ e^{\varphi A^{\alpha}} \partial_t u$,
and denoting $h^*(t) \coloneqq e^{\varphi(t) A^{\alpha}} h(t)$, it follows that~$u^*$ solves
\begin{subequations}\label{sys_u*}
\begin{align}
\partial_t u^* + \nu Au^* + e^{\varphi A^{\alpha}} B(u)
+ h^* - \frac{d \varphi}{dt} A^{\alpha}u^* &= 0,\label{difeq_u*}\\
u^*(0) &= u_0. \label{icu*}
\end{align}
\end{subequations}
Now, let $s \ge 0$ be another nonnegative number and multiply \eqref{difeq_u*} by $A^{2s}u^*$, formally we obtain 
\begin{align*}
&\quad\left( \partial_t u^*, A^{2s} u^* \right)_H + \nu \left( Au^*, A^{2s} u^* \right)_H\\
&= - \left( e^{\varphi A^{\alpha}} B(u) , A^{2s} u^* \right)_H 
- \left( h^*, A^{2s} u^* \right)_H + \frac{d \varphi}{dt} \left( A^{\alpha}u^* , A^{2s} u^* \right)_H .
\end{align*}
From the fact that $\left( e^{\varphi A^{\alpha}} B(u) , A^{2s} u^* \right)_H=\left( B(u) , A^{2s} e^{\varphi A^{\alpha}} u^* \right)_H$
and $ \left| \frac{d \varphi}{dt}  \right| \le 1$ for all $t \ge 0$, it follows
\begin{align}
&\frac{1}{2} \frac{d}{dt} \left| A^s u^*  \right|^2_H + \nu  \left| A^{s+ \frac{1}{2}} u^*  \right|^2_H \notag\\
\le& \left| \left( B(u), A^{2s} e^{\varphi A^{\alpha}} u^* \right)_H   \right| + \left| A^{s - \frac{1}{2}} h^* \right|_H
\left| A^{s + \frac{1}{2}} u^* \right|_H + \left| A^{s+\alpha-\frac{1}{2}} u^* \right|_H
\left| A^{s + \frac{1}{2}} u^* \right|_H .\label{norm_ineq}
\end{align}
Now, we find an appropriate bound for the term $\left| \left( B(u), A^{2s} e^{\varphi A^{\alpha}} u^* \right)_H   \right|$.
Recall the strictly increasing sequence $(\lambda_k)_{k\in\N_0}$ of all the distinct eigenvalues of
the Stokes operator~$A$ and the orthogonal projections $P_k:H\to P_kH$ onto the $\lambda_k$-eigenspace; see~\eqref{projPk} above.
We observe that for any $u \in H$, we may write
\begin{equation}\label{uPku}
u = \sum_{k \in \N_0}{P_ku}.
\end{equation}
\begin{remark}
Given nonnegative real numbers $s$, $\alpha$, and $\sigma$, $u\in\D\left( A^s e^{\sigma A^{\alpha}}\right)$, and $l\in\N_0$,
we have $P_l(A^{s} e^{\sigma A^{\alpha}} u)=\lambda_l^{s} e^{\sigma \lambda_l^{\alpha}} P_l u$,
and  $|u|_{\D\left( A^s e^{\sigma A^{\alpha}}\right)}^2
=\sum \limits_{k \in \N_0} e^{2\sigma \lambda_k^{\alpha}}\lambda_k^{2s} \left|P_k u\right|^2 $.
\end{remark}

From~\eqref{uPku} and Assumption~\ref{A:B}, we may write
\begin{align*}
&\quad\left(  B(u), A^{2s}e^{\varphi A^{\alpha}} u^* \right)_H =\sum_{(m,n,l) \in \N_0^3}
b \left( P_m u, P_n u, P_l(A^{2s} e^{\varphi A^{\alpha}} u^*) \right) \\
 &=\frac{1}{2}\sum_{(m,n,l) \in \N_0^3}
 \left( B(P_m u + P_n u),\,\lambda_l^{2s} e^{2\varphi \lambda_l^{\alpha}} P_l u \right)_H 
 = \sum_{ \begin{subarray}{l} m \in \N_0 \\ n < m \\  l \in   \FF_{n,\,m}^\bullet 
\end{subarray}} \left( B(P_m u + P_n u),\,\lambda_l^{2s} e^{2\varphi \lambda_l^{\alpha}} P_l u \right)_H \\
&= -\sum_{ \begin{subarray}{l} m \in \N_0 \\ n < m \\  l \in   \FF_{n,\,m}^\bullet   \end{subarray}}
b \left( P_n u,  \lambda_l^{2s} e^{2\varphi \lambda_l^{\alpha}} P_l u,  P_m u   \right)-\sum_{ \begin{subarray}{l} m \in \N_0 \\ n < m \\  l \in   \FF_{n,\,m}^\bullet   \end{subarray}}
b \left( P_n u,  \lambda_l^{2s} e^{2\varphi \lambda_l^{\alpha}} P_l u,  P_m u   \right). 
\end{align*}
%
%
Hence by Assumptions~\ref{A:A}, \ref{A:B} and~\ref{A:Bnml}, we can derive that
\begin{align*}
&\quad\left|\left(  B(u), A^{2s}e^{\varphi A^{\alpha}} u^* \right)_H \right|
\le 2C  \sum_{ \begin{subarray}{l} m \in \N_0 \\ n < m \\ 
l \in   \FF_{n,\,m}^\bullet \end{subarray}} 
\left|   P_n u \right|_{ L^{\infty} (\Omega,\,T\Omega) }
\lambda_l^{2s} e^{2\varphi \lambda_l^{\alpha}}\left| P_l u \right|_{ H^{1} (\Omega,\,T\Omega) }
\left|  P_m u \right|_{ L^{2} (\Omega,\,T\Omega) }\\
&\le 2C  \sum_{ \begin{subarray}{l} m \in \N_0 \\ n < m \\ 
l \in   \FF_{n,\,m}^\bullet \end{subarray}} 
\left|   P_n u \right|_{ L^{\infty} (\Omega,\,T\Omega) }
\lambda_l^{2s}\ex^{\varphi (\lambda_l^{\alpha}+\lambda_n^{\varsigma}+\lambda_m^{\alpha}+\beta)}
\left|A^{\frac{1}{2}}  P_l u \right|_{ L^{2} (\Omega,\,T\Omega) }
\left|  P_m u \right|_{ L^{2} (\Omega,\,T\Omega) }\\
&\le 2C  \sum_{ \begin{subarray}{l} m \in \N_0 \\ n < m \\ 
l \in   \FF_{n,\,m}^\bullet \end{subarray}} \ex^{\varphi \beta}
\left|   P_n u^* \right|_{ L^{\infty} (\Omega,\,T\Omega) }
\lambda_l^{2s+\frac{1}{2}}
\left|P_l u^* \right|_{ H }
\left|  P_m u^* \right|_{H}.
\end{align*}
From a suitable Agmon inequality (cf.~\cite{Temam95}, Section~2.3), it follows that
$\left|   P_n u^* \right|_{ L^{\infty} (\Omega,\,T\Omega) } 
\le   C_1\left|   P_n u^* \right|_{ L^{2} (\Omega,\,T\Omega) }^{\frac{4-d}{4} }  
\left|   P_n u^* \right|_{ H^{2} (\Omega,\,T\Omega) }^{\frac{d}{4} }  $  and 
\begin{equation} \label{Bu_A^2s_ineq}
\left|\left(  B(u), A^{2s}e^{\varphi A^{\alpha}} u^* \right)_H \right| \le C_2\ex^{\sigma \beta} 
\sum_{ \begin{subarray}{l} m \in \N_0 \\ n < m \\  l \in   \FF_{n,\,m}^\bullet \end{subarray}} 
{ \lambda_n^{\frac{d}{4}} \lambda_l^{2s + \frac{1}{2}}  
\left| P_n u^* \right|_H  \left| P_m u^* \right|_H  \left| P_l u^* \right|_H }. 
\end{equation}

\begin{remark}
Notice that the Agmon inequalities we find in~\cite[Section~2.3]{Temam95} concern the case~$\Omega$ is a subset of~$\R^d$. However they hold
also for a boundaryless manifold~$\CC$, because we can cover~$\CC$ by a finite number of charts and use a partition of unity argument.
Recall that the Sobolev spaces on a manifold may be defined by means of an atlas of~$\CC$ (cf.~\cite[Chapter~4, Section~3]{Taylor97}). They
hold also for smooth manifolds~$\Omega$ with smooth boundary~$\p\Omega$ (cf. the discussion after Equation~(4.11)
in~\cite[Chapter~4, Section~4]{Taylor97}).
\end{remark}

\begin{lemma} \label{lemmaBu} 
Suppose that the Assumptions~\ref{A:A}, \ref{A:B}, \ref{A:Bnml} and~\ref{A:Fnml} hold, and let the strictly increasing
 sequence of (nonrepeated) eigenvalues $(\lambda_k)_{k\in\N_0}$ of the Stokes operator~$A$ satisfy~\eqref{lambda_k}.
 Then, for any given
$ s > \frac{d + 2 \left(  \xi^{-1}+  2\zeta   -1 \right)}{4}$, there exists $C_B\in\R_0$ such that 
\begin{align*}
\left|\left(  B(u), A^{2s}e^{\varphi A^{\alpha}} u^* \right)_H \right| &\le C_B \left|  A^s u^*  \right|^2 
\left|  A^{s+ \frac{1}{2}} u^*  \right|,\qquad\mbox{ if }\; 4s \ge d+  2 \left(  \xi^{-1}+  2\zeta   +1 \right);\\
\left|\left(  B(u), A^{2s}e^{\varphi A^{\alpha}} u^* \right)_H \right|
&\le C_B \left|  A^s u^*  \right|^{\frac{6-\left( d-4s + 2\xi^{-1}+2  2\zeta   \right)}{4}}
\left|  A^{s+ \frac{1}{2}} u^*  \right|^{\frac{6+d-4s + 2\xi^{-1}+2  2\zeta  }{4}},\\
&\hspace*{12.75em}\mbox{ if }\; 4s <  d+ 2 \left(  \xi^{-1}+  2\zeta   +1 \right).\\
\end{align*}
Further, $C_B$ depends on $d$, $s$, $\sigma$, $\alpha$, $\beta$, $C_\FF$, $ \zeta $, $\rho$, and $\xi$.
\end{lemma}

\begin{proof}
From~\eqref{Bu_A^2s_ineq}, Assumption~\ref{A:Bnml} and Lemma~\ref{L:lambda_s}, it follows that 
\begin{align*}
\left|\left(  B(u), A^{2s}e^{\varphi A^{\alpha}} u^* \right)_H \right| 
\le K \sum_{\begin{subarray}{l} m \in \N_0 \\ n < m \\  l \in   \FF_{n,\,m}^\bullet \end{subarray} }
{ \lambda_n^{\frac{d}{4}} \lambda_l^{s + \frac{1}{2}}  \lambda_m^{s}  
\left| P_n u^* \right|_H  \left| P_m u^* \right|_H  \left| P_l u^* \right|_H } , 
\end{align*}
with $K=K(s,\,\sigma,\,\alpha,\,\beta,\,\lambda_1)$. Now we notice that for any triple $(m,n,l) \in \N_0^3$ with $n<m$ we have that
\[
l \in   \FF_{n,\,m}^\bullet\Leftrightarrow \left( B(P_nH+P_mH),\,P_lH\right)_H\ne\emptyset\Leftrightarrow m \in   \FF_{n,\,\bullet}^l;
\]
thus, by the Cauchy inequality, we obtain that
\begin{align*}
&\quad \left|\left(  B(u), A^{2s}e^{\varphi A^{\alpha}} u^* \right)_H \right|^2\\
&\le  K \Biggl(\sum_{ \begin{subarray}{l} m \in \N_0 \\ n < m \\  l \in   \FF_{n,\,m}^\bullet \end{subarray}}
 \lambda_n^{\frac{d}{4}}   \left| P_n u^* \right|_H   \lambda_m^{2s} 
\left| P_m u^* \right|^2_H    \Biggr)^{\frac{1}{2}}
\Biggl(\sum_{ \begin{subarray}{l} l \in \N_0 \\ n < m \\  m \in   \FF_{n,\,\bullet}^l \end{subarray}}
 \lambda_n^{\frac{d}{4}}   \left| P_n u^* \right|_H   \lambda_l^{2s+1} 
\left| P_l u^* \right|^2_H    \Biggr)^{\frac{1}{2}}.
\end{align*}
From Assumption~\ref{A:Fnml} we obtain
\begin{align*}
&\quad \left|\left(  B(u), A^{2s}e^{\varphi A^{\alpha}} u^* \right)_H \right|^2\\
&\le  K C_\FF \Biggl(\sum_{ n \in \N_0}
 \lambda_n^{\frac{d}{4}+ \zeta }   \left| P_n u^* \right|_H \Biggr) \Biggl(\sum_{ m \in \N_0} \lambda_m^{2s} 
\left| P_m u^* \right|^2_H    \Biggr)^{\frac{1}{2}}
\Biggl(\sum_{ l\in\N_0}
 \lambda_l^{2s+1} 
\left| P_l u^* \right|^2_H    \Biggr)^{\frac{1}{2}}.
\end{align*}

%
%
Now, again thanks to the Cauchy inequality, for $\gamma\in\R$ we find
\begin{align}
&\quad\left|\left(  B(u), A^{2s}e^{\varphi A^{\alpha}} u^* \right)_H \right|\notag\\  &\le  K C_{\FF}
\left(\sum_{n \in \N_0}{ \lambda_n^{\frac{d}{2}+  2\zeta  -2s-\gamma}  } \right) ^{\frac{1}{2}}
\left|A^{s+ \frac{\gamma}{2}}u^*\right|_H\left|A^{s}u^*\right|_H \left|A^{s+ \frac{1}{2}}u^*\right|_H. \label{Bu_A^2s_ineq2}
\end{align}
Since $s > \frac{d + 2 \xi ^{-1}+  4\zeta   -2}{4}$, we obtain 
$\frac{d}{2} - 2s +  2\zeta  < 1 -\xi^{-1}$. Thus, we may set
$\gamma \in \left(\frac{d}{2} - 2s + \xi^{-1}+  2\zeta   , 1  \right) $; which implies
$ \frac{d}{2} - 2s +  2\zeta   - \gamma < -\xi^{-1} $ and
$ \delta := \left( \frac{d}{2} - 2s +  2\zeta   - \gamma  \right)\xi <-1$.
From~\eqref{lambda_k}, it follows that 
\begin{equation}\label{sum_ndelta}
\sum_{n \in \N_0}{\lambda_n^{\frac{d}{2} - 2s +  2\zeta   - \gamma}}
\le \rho^{\frac{d}{2} - 2s +  2\zeta   - \gamma} \sum_{n \in \N_0}{n^{\delta}}
\eqqcolon C_{d,s,\rho,\xi, \zeta ,\gamma} < + \infty.
\end{equation}
and, choosing in particular $ \gamma =\bar\gamma\coloneqq \frac{d-4s + 2\left( \xi^{-1} +  2\zeta  +1 \right) }{4}$,
from~\eqref{Bu_A^2s_ineq2} and~\eqref{sum_ndelta},
it follows that 
\[
\left|\left(  B(u), A^{2s}e^{\varphi A^{\alpha}} u^* \right)_H \right| 
\le KC_{d,s,\rho,\xi, \zeta }\left|  A^{s+ \frac{\bar\gamma}{2}} u^*  \right|_H \left|  A^{s} u^*  \right|_H
\left|  A^{s+\frac{1}{2}} u^*  \right|_H.
\]
If $\bar\gamma \le 0$, that is if $4s \ge d + 2 \left( \xi^{-1} +  2\zeta  +1  \right)$, then 
\[
\left| \left( B(u^*), A^{2s} u \right)_H   \right| \le KC_{d,s,\rho,\xi}  \left|  A^{s} u^*  \right|_H^2
\left|  A^{s+\frac{1}{2}} u^*  \right|_H.
\]
If $\bar\gamma \in (0,1)$, that is if $d + 2\left(  \xi^{-1} +  2\zeta  -1  \right)  < 4s < d+ 2\left(  \xi^{-1}+  2\zeta   +1  \right) $, then by
an interpolation argument (cf. \cite{LioMag72-I}, Chapter~1), we can obtain that
\begin{align*}
\left|\left(  B(u), A^{2s}e^{\varphi A^{\alpha}} u^* \right)_H \right|
&\le KC_{d,s,\rho,\xi, \zeta } C_1  \left|  A^{s} u^*  \right|_H^{1+ (1-\bar\gamma)}
\left|  A^{s+\frac{1}{2}} u^*  \right|_H^{1+\bar\gamma} \\
&= KC_{d,s,\rho,\xi, \zeta } C_1 \left|  A^{s} u^*  \right|_H^{\frac{-d+4s-2 \xi^{-1}-  4\zeta   +6}{4}}
\left|  A^{s+\frac{1}{2}} u^*  \right|_H^{\frac{d-4s+2 \xi^{-1}+  4\zeta   +6}{4}},
\end{align*}
which completes the proof of the lemma.
\end{proof}

\subsection{Proof of Theorem~\ref{T:main}}\label{sS:proofT:main}
We look for $u$ in the form $u = e^{-\varphi(t) A^{\alpha}} u^*$ where $u^*$ solves \eqref{sys_u*}.
We will use Lemma~\ref{lemmaBu}, which suggests us to consider two cases. 
\subsubsection{The case $4s< d + 2\left( \xi^{-1}+  2\zeta   +1  \right)$. Existence}
We start by observing that
\[
\left|A^{s+\alpha-\frac{1}{2}} u^* \right|_H\le \lambda_1^{\alpha-\frac{1}{2}}\left|A^{s} u^* \right|_H,
\quad\mbox{if } \alpha\le\frac{1}{2},
\]
and, by an interpolation argument
\[
\left|A^{s+\alpha-\frac{1}{2}} u^* \right|_H
\le \left|A^{s} u^* \right|_H^{2(1-\alpha)}\left|A^{s+\frac{1}{2}} u^* \right|_H^{2\alpha-1},
\quad\mbox{if } \frac{1}{2}<\alpha<1.
\]
Next, since $4s> d + 2\left( \xi^{-1} +  2\zeta  -1  \right) $, we have $\frac{6 + d - 4 s + 2 \xi^{-1}+  4\zeta  }{4}< 2$. Thus, we can
set $ p = \frac{8}{6 + d - 4s + 2 \xi^{-1}+  4\zeta  } > 1$, and $q$ such that $ \frac{1}{p}+ \frac{1}{q} =1$, that is,
$ \frac{1}{q} = \frac{2- \left( d-4s + 2 \xi^{-1} +  4\zeta   \right)}{8}$.

From~\eqref{norm_ineq}, Lemma~\ref{lemmaBu},
and suitable Young inequalities, we derive that
\begin{align}
\frac{d}{dt} \left|A^s u^* \right|^2_H   + \frac{4\nu}{3}\left| A^{s+\frac{1}{2}} u^* \right|_H^2
&\le 2^q \left( \frac{3}{\nu} \right)^{\frac{q}{p}} C_B^q 
\left|  A^{s} u^*  \right|_H^{\left(\frac{-d+4s-2 \xi^{-1} -  4\zeta  +6}{4} \right)q} + \frac{3}{\nu} 
\left|A^{s-\frac{1}{2}} h^* \right|^2_H\notag\\
&\quad  + \left|A^{s+\alpha-\frac{1}{2}} u^* \right|_H\left|A^{s+\frac{1}{2}} u^* \right|_H.\label{dAsu*}
\end{align}
Notice that  in the case $\alpha\in(0,\,\frac{1}{2}]$ we have
\[
 \left|A^{s+\alpha-\frac{1}{2}} u^* \right|_H\left|A^{s+\frac{1}{2}} u^* \right|_H
 \le C_{\nu}\left|A^{s} u^* \right|_H^2+\frac{\nu}{3}\left|A^{s+\frac{1}{2}} u^* \right|_H^2,
\]
and in the case $\alpha\in(\frac{1}{2},\,1)$ we have
\begin{align*}
 \left|A^{s+\alpha-\frac{1}{2}} u^* \right|_H\left|A^{s+\frac{1}{2}} u^* \right|_H
 &\le C_{\nu,\alpha}\left|A^{s} u^* \right|_H^{2(1-\alpha)\frac{2}{3-2\alpha}}
 +\frac{\nu}{3}\left|A^{s+\frac{1}{2}} u^* \right|_H^2\\
 &\le C_{\nu,\alpha}\left(\left|A^{s} u^* \right|_H+1\right)^2
 +\frac{\nu}{3}\left|A^{s+\frac{1}{2}} u^* \right|_H^2\\
 &\le 2C_{\nu,\alpha}\left(\left|A^{s} u^* \right|_H^2+1\right)
 +\frac{\nu}{3}\left|A^{s+\frac{1}{2}} u^* \right|_H^2,
\end{align*}
because $0<\frac{4(1-\alpha)}{3-2\alpha}<2$.

Next we observe that $\left(\frac{-d+4s-2 \xi^{-1} -  4\zeta  +6}{4} \right)q = 2+ q>3$, and from Proposition~\ref{P:ab_s} it follows
$\left|  A^{s} u^*  \right|_H^{2+q}+1\le\left(\left|  A^{s} u^*  \right|_H^{2}+1\right)^{\frac{2+q}{2}}$.
Therefore, from~\eqref{dAsu*}, we can obtain
\begin{align*}
&\quad\frac{d}{dt} \left|A^s u^* \right|^2_H   + \nu\left| A^{s+\frac{1}{2}} u^* \right|_H^2\\
&\le K_1
\left|  A^{s} u^*  \right|_H^{2+q} + \frac{3}{\nu} 
\left|A^{s-\frac{1}{2}} h^* \right|^2_H
+ K_2(\left|A^{s} u^* \right|_H^2+1)\\
&\le (K_1+K_2)
\left(\left|  A^{s} u^*  \right|_H^{2}+1\right)^{\frac{2+q}{2}} + \frac{3}{\nu} 
\left|A^{s-\frac{1}{2}} h^* \right|^2_H,
\end{align*}
with $K_1+K_2$ depending on $\nu$, $\lambda_1$, $d$, $s$, $\sigma$, $\alpha$, $\beta$, $\rho$, $\xi$, $ \zeta $, and $C_\FF$.\\
Now, setting $K_3\coloneqq K_1+K_2+\frac{3}{\nu} 
\left|A^{s-\frac{1}{2}} h^* \right|_{L^\infty((0,\,+\infty),\,H)}^2$, we arrive to
\begin{equation}\label{dAsu*2}
 \quad\frac{d}{dt} \left|A^s u^* \right|^2_H   + \nu\left| A^{s+\frac{1}{2}} u^* \right|_H^2
\le K_3
\left(\left|  A^{s} u^*  \right|_H^{2}+1\right)^{\frac{2+q}{2}}
\end{equation}
and, in particular, to
\[
\frac{d}{dt} y
\le K_3 y^{\frac{{2+q}}{2}},\quad\mbox{with }y(t)\coloneqq \left|A^s u^*(t) \right|^2_H + 1,
\]
that is,
$ \frac{d}{dt} y^{\gamma} \ge \gamma K_3$ with $ \gamma \coloneqq 1- (\frac{2+q}{2})= -\frac{q}{2}<0 $.
Integrating over the interval $(0,\,t)$, it follows that $ y^{\gamma}(t) \ge y^{\gamma}(0) - (\frac{q}{2}) K_3t $.
If we set $T^*$ such that $ (\frac{q}{2}) K_{B, \nu} T^* \le (\frac{1}{2}) y^{\gamma}(0)$, that is if
$ T^* \le \frac{y^{\gamma}(0)}{q K_{B, \nu}}$, then  $  y^{-\gamma}(t)  \le 2 y^{-\gamma}(0) $, for all $t \in \left[ 0, T^* \right]$.
Thus, we obtain 
\[
 \left|A^s u^*(t) \right|^2_H + 1 \le 4^{\frac{1}{q}} \left( \left|A^s u(0) \right|^2_H + 1 \right)
 \quad \text{for all }t \in \left[0, T^* \right],
\]
from which, together with $u(0)=u_0\in\D(A^s)$ and \eqref{dAsu*2}, we can conclude that
\begin{equation}\label{u^*}
u^* \in  L^{\infty} \left( \left(0,T^* \right), D(A^s)  \right) \cap  L^{2} \left( \left(0,T^* \right), D(A^{s+\frac{1}{2}})  \right) 
\end{equation}
which implies~\eqref{reg_u}.
\subsubsection{The case $4s \ge d + 2\left( \xi^{-1} +  2\zeta  +1  \right)$. Existence}
Using the corresponding inequality from~\ref{lemmaBu}, it is straightforward to check that all the arguments from
the first case, $4s < d + 2\left( \xi^{-1} +  2\zeta  +1  \right)$,
can be repeated by taking $p = q =2$. We will arrive again to the conclusions \eqref{u^*}, and~\eqref{reg_u}.

\subsubsection{Uniqueness} It remains to check the uniqueness of $u$. Let $v$ be another solution for~\eqref{sys_u}, and set $\eta=v-u$.
We start by noticing that, from~\eqref{reg_u}, with nonnegative
$(s,\,\sigma,\,\alpha)\in[0,\,+\infty)^3$, we have in particular that $u$ is a weak solution:
\[
 u \in L^{\infty} \left(  (0,T^*), H \right)
\cap  L^{2} \left(  (0,T^*), D  \left( A^{\frac{1}{2}} \right) \right).
\]
In the case $d=2$, it is well known that the uniqueness of $u$ will follow from the estimate
\begin{align*}
 |(B(v)-B(u),\,\eta)_H|
&=|b(\eta,\,u,\,\eta)|\le|\eta|_{L^4(\Omega,\,T\Omega)}|\eta|_{L^4(\Omega,\,T\Omega)}|u|_{H^1(\Omega,\,T\Omega)}\\
&\le C|\eta|_{H^{\frac{1}{2}}(\Omega,\,T\Omega)}^2|u|_{H^1(\Omega,\,T\Omega)}
\le C_1|\eta|_H|A^{\frac{1}{2}}\eta|_H|A^{\frac{1}{2}}u|_H 
\end{align*}
(see, e.g.,~\cite[Chapter~3, Section~3.3, Theorem~3.2]{Temam01}).

In the case $d=3$. Since $s>\frac{d-2}{4}=\frac{1}{4}$, again from~\eqref{reg_u}, we also have that
\[
 u \in L^{\infty} \left(  (0,T^*), D  \left( A^{s_1} \right) \right) \subseteq L^{\infty} \left(  (0,T^*), H^{2s_1}(\Omega,\,\R^3) \right)
 \subset L^{r_1} \left(  (0,T^*), L^{r_2}(\Omega,\,\R^3) \right)
\]
with $s_1<s$ and $s_1\in(\frac{1}{4},\,\frac{1}{2}]$, $r_1>1$ and $r_2=\frac{2d}{d-4s_1}>3$,
by the Sobolev embedding Theorem~(cf.~\cite[Section~4.4, Corollary~4.53]{DemengelDem12}).
Now, the uniqueness of $u$
follows from the fact that for $r_1$ big enough we have that $\frac{2}{r_1}+\frac{d}{r_2}\le1$,
and from~\cite[Chapter~1, Section~6.8, Theorem~6.9]{Lions69}).
\qed

\begin{remark}
 For simplicity we have restricted ourselves to the above formal computations,
 but those computations will hold for the Galerkin approximations based on the eigenfunctions of~$A$,
 which means that they can be made rigorous. See, for example, \cite[Chapter~1, Section~6.4]{Lions69} and \cite[Chapter~3, Section~3]{Temam01} 
\end{remark}

\section{Considering repeated eigenvalues}\label{S:repeated}
In some cases it will be more convenient to work with the sequence $(\underline\lambda_k)_{k\in\N_0}$ of repeated eigenvalues.
In that case we have to adjust our assumptions to obtain the corresponding version of the Theorem~\ref{T:main}. Consider the
system of eigenfunctions $\{W_k\mid k\in\N_0\}$.

\begin{assumption}\label{A:Bnml_r}
There are real numbers $\alpha>0$ and $\beta\ge0$, such that for all triples $(n,\,m,\,l)\in\N_0^3$
\[
\begin{array}{l}
        (B(W_n+W_m),\,W_l)_H\ne 0,
       \end{array}
\quad\mbox{implies}\quad
\underline\lambda_l^\alpha\le \underline\lambda_n^\alpha+\underline\lambda_m^\alpha+\beta.
\]
\end{assumption}

For given $(n,\,m,\,l)\in\N_0^3$, we define the sets
\begin{align*}
 \underline\FF_{n,\,m}^\bullet&\coloneqq\left\{k\in\N_0\mid(B(W_n+W_m),\,W_k)_H\ne 0,\mbox{ with } n<m \right\};\\
 \underline\FF_{n,\,\bullet}^l&\coloneqq\left\{k\in\N_0\mid(B(W_n+W_k),\,W_l)_H\ne 0,\mbox{ with } n<k \right\}.
\end{align*}

\begin{assumption}\label{A:Fnml_r}
There are $C_\FF\in\N_0$ and $ \zeta \in[0,\,+\infty)$, such that for all $n \in \N_0$ we have
\[
\sup_{(m,l) \in \N_0^2} \left \{  {\rm card}(\underline\FF_{n,\,m}^\bullet),{\rm card}(\underline\FF_{n,\,\bullet}^l) \right \}   \le C_\FF\underline\lambda_n^{ \zeta }.
\]
\end{assumption}

\begin{theorem} \label{T:main_r}
 Suppose that the Assumptions~\ref{A:A}, \ref{A:B}, \ref{A:Bnml_r} and~\ref{A:Fnml_r} hold, and let the increasing
 sequence of (repeated) eigenvalues $(\underline\lambda_k)_{k\in\N_0}$ of the Stokes operator~$A$ satisfy, for some positive
 real numbers $\rho,\,\xi$,
 \begin{equation*}
  \underline\lambda_k>\rho k^\xi,\quad\mbox{for all }k\in\N_0.
 \end{equation*}
Further, let us be given $\alpha\in(0,\,1)$ as in Assumption~\ref{A:Bnml_r}, $C_\FF$ and $ \zeta \ge0$ as in
Assumption~\ref{A:Fnml_r}, $s>\frac{d+2(\xi^{-1}+  2\zeta  -1)}{4}$, $\sigma>0$,
$h\in L^{\infty}(\R_0,\,\D(A^{s-\frac{1}{2}}e^{\sigma A^{\alpha}}))$, and $u_0\in\D(A^s)$.\\
Then, there are $T^*>0$ and a unique solution
\begin{equation*}
 u \in L^{\infty} \left(  (0,T^*), D  \left( A^s e^{\sigma A^{\alpha}} \right) \right) 
\cap  L^{2} \left(  (0,T^*), D  \left( A^{s+\frac{1}{2}} e^{\sigma A^{\alpha}} \right) \right)
\end{equation*}
for the Navier Stokes system \eqref{sys_u}.\\
Further, $T^*$ depends on
the data
$\left(|h|_{L^{\infty} (\R_0,\, D \left( A^{s- \frac{1}{2}}e^{\sigma A^{\alpha}} \right)} ,\, \left| A^s u_0  \right|_H  \right)$
and also on the constants $\nu$, $d$, $s$, $\sigma$, $\alpha$, $\rho$, $\xi$, $ \zeta $, and $C_\FF$.
\end{theorem}
The proof can be done following line by line that of Theorem~\ref{T:main}.

\begin{remark}\label{R:linf_bound}
If we can find a bound $ \left| P_nu  \right|_{L^{\infty}   \left(\Omega,\,T\Omega \right) } 
\le C \lambda_n^{\theta} |P_nu|_H$  with $ \theta < \frac{d}{4} $ and $C$ independent of $n$, then we can take  $\theta$ in the place of $ \frac{d}{4}$ in \eqref{Bu_A^2s_ineq}. As a corollary, 
we can replace $d$ by $4\theta$ in Theorem~\ref{T:main}, provided $s$ satisfies $s \ge 0$
in the case $d=2$ and $s > \frac{1}{4}$ in the case $d=3$, in order to guarantee the uniqueness of the solution.
The analogous conclusion holds for Theorem~\ref{T:main_r}, if we can find a bound 
$ \left| W_n  \right|_{L^{\infty}\left(\Omega,\,T\Omega \right) } 
\le C \lambda_n^{\theta}$.
\end{remark}

\begin{remark}
In some situations like in the case of general Navier boundary conditions it may be useful to
split the Stokes operator~$\Pi\Delta$ as $\Pi\Delta=A+C$ (cf.~\cite[Chapter~4, Section~4.2]{Rod-Thesis08}),
or it may be interesting to consider an additional linear external forcing
(like a Coriolis forcing as in~\cite{CaoRammahaTiti99}). In these cases we will have the system
\[
 \dot u+B(u,\,u)+A u+Cu+h=0,\quad u(0,\,x)=u_0(x),
\]
instead of~\eqref{sys-u-evol}.
Notice that Theorems~\ref{T:main} and~\ref{T:main_r} will hold in these cases provided we have the estimate
$(Cu,\,A^{2s}u)_{V',\,V}\le C_1\norm{A^{s}u}{H}\norm{A^{s+\frac{1}{2}}u}{H}$. A better estimate
holds in the case of the two-dimensional Navier--Stokes equation under the action of a Coriolis force~$\tilde Cu$:
from~\cite[Lemma~1]{CaoRammahaTiti99}, for $s>\frac{1}{2}$ it holds
$(Cu,\,A^{2s}u)_{V',\,V}=(A^{s+\frac{1}{2}}Cu,\,A^{s-\frac{1}{2}}u)_{H}\le C_1\norm{A^s}{H}\norm{A^{s-\frac{1}{2}}}{H}$,
with $Cu\coloneqq\Pi\tilde Cu$.
\end{remark}

\section{Examples}\label{S:examples}
We start by revisiting the cases where $\Omega$ is the Torus~$\T^d$ and the Sphere~$\Sp^2$. Then we
give some new examples in two dimensions, namely the cases of Hemisphere, Rectangle and Cylinder under Lions boundary conditions.

\subsection{Torus} \label{sS:Torus}
We consider the torus $\T^d = \Pi_{i=1}^d\Sp^1\sim (0, 2 \pi]^d$, $d\in\{2,\,3\}$. This case corresponds to the case where we take
periodic boundary conditions in~$\R^d$ with period~$2 \pi$ in each direction $x_i$, $i\in\{1,\,\dots,\,d\}$.
We also assume that the average~$\int_{\T^2}u(t)\,\ed\T^d$ vanishes for (a.e.) $t\ge0$
(cf.~\cite[Chapter~II, eq.~(2.5)]{FoiManRosaTem01}, \cite[Section~2.1]{AgraSary05}).  
In this case the Navier--Stokes system can be rewritten as an evolutionary equation
in the space of divergence free and zero averaged vector
fields~$H=\{u\in L^2(\T^d,\,T\T^d)\sim L^2(\T^d,\,\R^d)\mid \diver u=0\mbox{ and }\int_{\T^2}u\,\ed\T^d=0\}$, with the spaces $V$ and
$\D(A)$, defined in Section~\ref{sS:evol_sys} given by~$V=H\cap H^1(\T^d,\,T\T^d)$ and~$\D(A)=H\cap H^2(\T^d,\,T\T^d)$.

We will show that in this case we can take $\alpha=\frac{1}{2}$,
$\xi=\frac{1}{2}$, and $\zeta=0$ in Theorem~\ref{T:main_r}, and $\theta=0$ in Remark~\ref{R:linf_bound}. That is, we can take~$s>0$,
in Theorem~\ref{T:main_r}.

To simplify the writing we will denote the usual Euclidean scalar product~$(u,\,v)_{\R^d}$ in~$\R^d$ by~$u\cdot v\coloneqq\sum_{i=1}^d u_iv_i$.
It is well known that a vector field can be written as
\[
 u=\sum_{\begin{subarray}{l} k\in\Z^d\setminus\{0_d\}\end{subarray}}u_k\ex^{\im k\cdot x},
\]
where $0_d$ stands for the zero element $(0,\,\dots,\,0)\in \R^d$, $\mathbf i\sim 0+1\mathbf i$ is the imaginary complex unit,
and the coefficients satisfy
$k\cdot u_k=0$ and $u_{-k}=\overline{u_k}$, where the overline stands for the complex conjugate.
The condition $k\cdot u_k=0$ comes from the divergence free condition,
and~$u_{-k}=\overline{u_k}$ comes from the fact that $u$ is a function with (real) values in $\R^3$.
Thus
\[
 u=\sum_{\begin{subarray}{l} k\in\Z^d;\;k>0_d\end{subarray}}\Re(u_k)\cos(k\cdot x)-\Im(u_k)\sin(k\cdot x),
\]
where $k>0_d$ is understood in the lexicographical order, that is either $k_1>0$, or $k_1=0 \mbox{ and } k_2>0$, or
$(k_1,\,k_{d-1})=(0,\,0) \mbox{ and } k_d>0$, and that a basis of vector fields in $H$ is given by
\[
 \WW=\{w_k^j\cos(k\cdot x),\,w_k^j\sin(k\cdot x)\mid k\in\Z^d,\,k>0_d\mbox{ and }j\in\{1,\,d-1\}\}
\]
where for each $k\in\Z^d,\,k>0_d$, $\{w_k^1,\,w_k^{d-1}\}$
is a basis for the orthogonal space $\{k\}^\bot$ of $\{k\}$, in $\R^d$. That is, $\linspan\{w_k^1,\,w_k^{d-1}\}=\{k\}^\bot$
(cf.~\cite[Chapter~6, Section~1]{Rod-Thesis08} for the case $d=2$). Moreover we may choose the vectors~$w_k^j$ so
that the basis above is orthonormal, that is, we can write
\[
 u=\sum_{\begin{subarray}{l} k\in\Z^d;\;k>0_d\\ j\in\{1,d-1\}\end{subarray}}u_{k,j}^c w_k^j\cos(k\cdot x)+u_{k,j}^s w_k^j\sin(k\cdot x).
\]
Since the cardinality of $\{k\in\Z^d\mid k>0_d\}\times\{1,\,d-1\}$ is equal to that of $\N_0$ we could write the previous sum as
$u=\sum_{k\in\N_0}u_kW_k$, as in the preceding text (cf. Section~\ref{sS:gevrey}). However we can check the Assumptions~\ref{A:A}, \ref{A:B},
\ref{A:Bnml_r}, and~\ref{A:Fnml_r} without performing that writing explicitly.

\noindent
{\bf Checking Assumptions~\ref{A:A} and~\ref{A:B}.} Assumptions~\ref{A:A} is well known to hold under
periodic boundary conditions. The same holds for all the points in Assumption~\ref{A:B} (cf.~\cite[Section~2.3]{Temam95});
we only check the first one, or taking into account
Remark~\ref{R:linf_bound}, we check that $B\left(w_n^j\cos(n \cdot x),\,w_n^j\cos(n \cdot x)\right) =B\left(w_n^j\sin(n \cdot x),\,w_n^j\sin(n \cdot x)\right)=0$.
Indeed from $w_n^j\cdot n=0$, it follows
\begin{align*}
 B(w_n^j\cos(n \cdot x),\,w_n^j\cos(n \cdot x))&=-\Pi(w_n^j\cdot n)\cos(n \cdot x)\sin(n \cdot x)w_n^j=0,\\
 B(w_n^j\sin(n \cdot x),\,w_m^j\sin(n \cdot x))&=\Pi(w_n^j\cdot n)\sin(n \cdot x)\cos(n \cdot x)w_n^j=0,
\end{align*}
where $\Pi$ stands for the orthogonal projection in $L^2(\T^d,\,T\T^d)$ onto $H$.

\noindent
{\bf Checking Assumptions~\ref{A:Bnml_r} and~\ref{A:Fnml_r}.} We proceed as follows: first we obtain 
\begin{align*}
 B(w_n^j\cos(n \cdot x),\,w_m^i\cos(m \cdot x))&=-\Pi(w_n^j\cdot m)\cos(n \cdot x)\sin(m \cdot x)w_m^i,\\
 B(w_n^j\cos(n \cdot x),\,w_m^i\sin(m \cdot x))&=\Pi(w_n^j\cdot m)\cos(n \cdot x)\cos(m \cdot x)w_m^i,\\
 B(w_n^j\sin(n \cdot x),\,w_m^i\cos(m \cdot x))&=-\Pi(w_n^j\cdot m)\sin(n \cdot x)\sin(m \cdot x)w_m^i,\\
 B(w_n^j\sin(n \cdot x),\,w_m^i\sin(m \cdot x))&=\Pi(w_n^j\cdot m)\sin(n \cdot x)\cos(m \cdot x)w_m^i,
\end{align*}
from which we can find that
\begin{align*}
&\quad\; B(w_n^j\cos(n \cdot x) + w_m^i\cos(m \cdot x))\\ &= B(w_n^j\cos(n \cdot x),\,w_m^i\cos(m \cdot x))
+B(w_m^i\cos(m \cdot x),\, w_n^j\cos(n \cdot x))\\
&\quad+B(w_n^j\cos(n \cdot x),\,w_n^j\cos(n \cdot x))+B(w_m^i\cos(m \cdot x),\,w_m^i\cos(m \cdot x))\\
&=-\Pi w_m^i(w_n^j\cdot m)\cos(n \cdot x)\sin(m \cdot x) -\Pi w_n^j(w_m^j\cdot n)\cos(m \cdot x)\sin(n \cdot x)\\
&=\frac{1}{2}\Pi(-w_m^i(w_n^j\cdot m)-w_n^j(w_m^j\cdot n))\sin((m+n)\cdot x)\\
&\quad+\frac{1}{2}\Pi(-w_m^i(w_n^j\cdot m)+w_n^j(w_m^j\cdot n))\sin((m-n)\cdot x),
\end{align*}
then, it is straightforward to check that $B(w_n^j\cos(n \cdot x) + w_m^i\cos(m \cdot x))$ is orthogonal in $L^2(\T^d,\,T\T^d)$
to all the elements in $\WW$
except those in $$\{w_{m+n}^j\sin\left((m+n) \cdot x \right),\,w_{[m-n]}^j\sin\left(([m-n]) \cdot x\right)\mid j\in\{1,\,d-1\}\},$$
where we denote
\[
[m-n]=\left\{
\begin{array}{ll}m-n&\mbox{if }m-n>0_d\\ n-m&\mbox{if }n-m>0_d\mbox{ or }n-m=0_d\end{array}\right..
\]
In other words, we can conclude that $(B (w_n^j\cos(n \cdot x) + w_m^i\cos(m \cdot x)),\,v)_H\ne0$ only if
$$v\in\linspan\{w_{m+n}^j\sin\left((m+n) \cdot x \right),\,w_{[m-n]}^j\sin\left(([m-n]) \cdot x\right)\mid j\in\{1,\,d-1\}\}.$$
Analogously, we can conclude that $(B(w_n^j\sin(n \cdot x) + w_m^i\sin(m \cdot x)),\,v)_H\ne0$ only if
$$v\in\linspan\{w_{m+n}^j\sin\left((m+n) \cdot x \right),\,w_{[m-n]}^j\sin\left(([m-n]) \cdot x\right)\mid j\in\{1,\,d-1\}\}.$$
Besides that $(B(w_n^j\sin(n \cdot x) + w_m^i\cos(m \cdot x)),\,v)_H\ne0$ only if
$$v\in\linspan\{w_{m+n}^j\cos \left((m+n) \cdot x \right),\,w_{[m-n]}^j\cos\left(([m-n]) \cdot x\right)\mid j\in\{1,\,d-1\}\};$$ 
and that $(B(w_n^j\cos(n \cdot x) + w_m^i\sin(m \cdot x)),\,v)_H\ne0$ only if
$$v\in\linspan\{w_{m+n}^j\cos\left((m+n) \cdot x \right),\,w_{[m-n]}^j\cos\left(([m-n]) \cdot x\right)\mid j\in\{1,\,d-1\}\}.$$
Therefore we can conclude that ${\rm card}(\underline\FF_{n,\,m}^\bullet)\le4$ and that
necessarily ${\rm card}(\underline\FF_{n,\,\bullet}^l)\le4$. That is, we can take $C_\FF=4$ and $\zeta=0$ in Assumption~\ref{A:Fnml_r}.

Assumption~\ref{A:Bnml_r} follows from the fact that the eigenvalue associated
to $w_n^j\sin(n \cdot x)$ and $w_n^j\cos(n \cdot x)$,
is given by $\norm{n}{\R^d}^2=n\cdot n$, and
by the triangle inequality,   
$\norm{n\pm m}{\R^d}\le\norm{n}{\R^d}+\norm{m}{\R^d}$,
which implies that Assumption~\ref{A:Bnml_r} holds with
$\alpha=\frac{1}{2}$ and $\beta=0$.

\noindent
{\bf Looking for the value~$\theta$ in Remark~\ref{R:linf_bound}.}
From $|w_n^j\sin(n \cdot x)|_{L^\infty(\T^d,\,T\T^d)}\le|w_n^j|$ and
$|w_n^j\sin(n \cdot x)|_{L^2(\T^d,\,T\T^d)}=1$, we have
$|w_n^j|^2=|\sin(n \cdot x)|_{L^2(\T^d,\,T\T^d)}^{-2}=\pi^{-d}$ and
$|w_n^j\sin(n \cdot x)|_{L^\infty(\T^d,\,T\T^d)} \le \pi^{-\frac{d}{2}}$, and similarly 
$|w_n^j\cos(n \cdot x)|_{L^\infty(\T^d,\,T\T^d)} \le \pi^{-\frac{d}{2}}$.
Hence, we can take $\theta=0$ in Remark~\ref{R:linf_bound}.

\noindent
{\bf Asymptotic behavior of the (repeated) eigenvalues.}
From~\cite[Chapter~II, Section~6]{FoiManRosaTem01}
we know that the asymptotic behavior of the (repeated) eigenvalues of the Stokes operator in $\T^d$
satisfy $\lambda_k\sim \lambda_1k^\frac{2}{d}$ and more precisely
\[
\lim_{k\to+\infty}\frac{\lambda_k}{\lambda_1k^\frac{2}{d}}\eqqcolon \zeta> 0;
\]
then in particular there is $k_0\in\N_0$ such that $\frac{\lambda_k}{\lambda_1k^\frac{2}{d}}\ge\frac{\zeta}{2}$ for all $k>k_0$,
which implies that for all $k\in\N_0$ we have $\lambda_k>\rho k^{\frac{2}{d}}$ if
$\rho<\lambda_1\min\{\frac{\zeta}{2},\,\zeta_0\}$, with $\zeta_0\coloneqq\min_{k\le k_0}\frac{\lambda_k}{\lambda_1k^\frac{2}{d}}$.
That is, we can take $\xi=\frac{2}{d}$ in Theorem~\ref{T:main_r}.

\noindent
{\bf Conclusion.}
Taking into account Remark~\ref{R:linf_bound}, we conclude that Theorem~\ref{T:main_r} holds
with~$\alpha=\frac{1}{2}$ and~$s>\frac{d-2}{4}$. This improves the results in~\cite{FoiasTemam89,Liu92},
from whose we already knew that $s$ could be taken
in $[\frac{1}{2},\,+\infty)$ for $d=2$, and in $\{\frac{1}{2}\}\cup[\frac{3}{4},\,+\infty)$ for $d=3$.

\subsection{Sphere} \label{sS:Sphere}
Let $\Omega = \mathbb S^2\coloneqq\{(x_1,\,x_2,\,x_3)\in\R^3\mid x_1^2+x_2^2+x_3^2=1\}$ be
the two-dimensional sphere with the Riemannian metric induced from the usual Euclidean metric in $\R^3$.

In this case we can write the Navier--Stokes system as an evolutionary equation in the
space~$H \coloneqq \{u \in L^2 (\Omega,\,T \Omega) \mid \nabla \cdot u =0  \} \cap \{\nabla^\bot\psi\mid\psi\in H^1(\Sp^2,\,\R)\}$, with
$V \coloneqq H\cap H^1 (\Omega,\,T\Omega)$ and $\D(A) \coloneqq H\cap H^2 (\Omega,\,T \Omega)$ (cf.~\cite[Section~5.6]{Rod-Thesis08},
\cite[Section~2]{CaoRammahaTiti99}).

\begin{remark}\label{R:culr-def}
Notice that in~\cite[Section~2]{CaoRammahaTiti99} and~\cite[Section~5.6]{Rod-Thesis08} the definitions and notations of the curl of a
function~$f$ are different;
in the former reference it is denoted~$\mathop{\rm Curl}f$ and in the latter~$\nabla^\perp f$;
we shall show that $\mathop{\rm Curl}f=-\nabla^\perp f$ in the Appendix, Section~\ref{sApx:curl}.
\end{remark}

In this case we will use Theorem~\ref{T:main} and Remark~\ref{R:linf_bound}, and show that there we can take $\theta = \frac{1}{4}$,
$\xi = 2$, $\zeta = \frac{1}{2}$, $\alpha=\frac{1}{2}$
and $s > \frac{1}{2}$. In particular we recover the result in~\cite{CaoRammahaTiti99}.

The complete system of eigenfunctions and respective eigenvalues for $A=-\nu\Pi\Delta$, in $H$, is presented
in \cite[Section~2]{CaoRammahaTiti99}, and it is given by
\begin{equation}\label{completesys-Sp}
 \{Z^m_n (\vartheta, \phi)
=\lambda_n^{-\frac{1}{2}} \nabla^{\perp} Y^m_n (\vartheta, \phi)\mid n\in\N_0, m\in\Z,\mbox{ and }|m|\le n\}
\end{equation}
where $\vartheta \in (0, \pi)$, $\phi \in (0, 2\pi)$,
and  for each $Y^m_n (\vartheta, \phi) \coloneqq  C^m_n e^{im \phi}P^m_n(\cos \vartheta)$ is
a normalized eigenfunctions of the Laplacian in $L^2(\Sp^2,\,\R)$ associated with the eigenvalue $\lambda_n = n(n+1)$, with
$C^m_n \coloneqq \left( \frac{2n+1}{4 \pi} \frac{(n - |m|)!} {(n - |m|)!}  \right)^{\frac{1}{2}}$ and $P^m_n$ is the
Ferrers' associated Legendre function of the first kind 
\begin{equation}\label{assocLegfunt}
P_n^m (x) \coloneqq \frac{\left( 1 - x^2 \right)^{\frac{m}{2}}}{2^n n!} \frac{\ed^{n+m}(x^2-1)^n}{\ed x^{n+m}} ,
\quad P_n^{-m}(x) \coloneqq P_n^{m}(x)
\end{equation}
for $m \in \{ k \in \N: k \le n \}$, defined for $|x| \le 1$. For further details on these functions we refer
to~\cite[Chapter~XV, Section~15.5]{WhittakerWatson69}.

For any $(u,v,w) \in P_nH \times P_mH \times P_lH$, there are scalar functions $(\psi_u, \psi_v, \psi_w)$
called  stream functions associated with $(u,v,w)$ respectively such that 
\[
u = - \nabla^{\perp} \psi_u,\quad
v = - \nabla^{\perp} \psi_v \quad
w = - \nabla^{\perp} \psi_w
\]
where 
\[
\psi_u = \sum_{|i|\le n }{\psi^i_u Y_n^i}, \quad
\psi_v= \sum_{|j|\le m }{\psi_v^j Y_m^j}, \quad
\psi_w = \sum_{|k|\le l}{\psi_w^k Y_l^k},
\]
and  $\psi_u^i =\overline{\psi_u^{-i}} $, $\psi_v^j =\overline{\psi_v^{-j}} $, and $\psi_w^k =\overline{\psi_w^{-k}} $.

\noindent
{\bf Checking Assumptions~\ref{A:A} and~\ref{A:B}.} Assumption~\ref{A:A} and the estimates in Assumption~\ref{A:B} follow straightforward.
Now we show that $B(u)=0$ if $u$ is an eigenfunction.
From the discussion after Corollary~5.6.3 in~\cite[Chapter~5, Section~5.6]{Rod-Thesis08} we
have that $\nabla^\perp\cdot (B(u))=g(\nabla\Delta^{-1}\nabla^\perp\cdot u,\,\nabla^\perp\nabla^\perp\cdot u)$, where~$g(\cdot,\,\cdot)$
is the scalar product
in~$T\Sp^2$ inherited from the Euclidean scalar product in~$\R^3$ and $\Delta^{-1}f$
denotes the solution $g$ of the Poisson system $\Delta g=f,\; g\rest{\p\Omega}=0$. If $u$ is an eigenfunction from~\eqref{completesys-Sp} with
associated eigenvalue~$\lambda_u$ and associated stream function~$\psi_u$, then $\nabla^\perp\cdot u=\Delta \psi_u$, and we
find~$\nabla^\perp\cdot B(u)
=g(\nabla \psi_u,\,\nabla^\perp\Delta \psi_u)=\lambda_u g(\nabla \psi_u,\,\nabla^\perp \psi_u)=0$, this implies that
$AB(u)=\Delta B(u)=\nabla^\perp\nabla^\perp\cdot (B(u))=0$, and necessarily $B(u)\in H$ is orthogonal to all eigenfunctions
in~\eqref{completesys-Sp}, $(B(u),Z_n^m)_H=\lambda_n^{-1}(B(u),AZ_n^m)_H=\lambda_n^{-1}(AB(u),Z_n^m)_H=0$, that is, $B(u)=0$.

Finally, for the skew-symmetry property $b(u,\,v,\,w)=-b(u,\,w,\,v)$ we refer to~\cite[Section~8, Equation~(59)]{Arnold66}
\cite[Chapter~5, Corollary~5.5.2]{Rod-Thesis08}.

\noindent
{\bf Checking the Assumptions~\ref{A:Bnml} and~\ref{A:Fnml}.}
Following~\cite[Section~3, Lemma~2]{CaoRammahaTiti99} (cf.~\cite[Chapter~5, corollary 5.6.3]{Rod-Thesis08}
\cite[Section~9, Equation~(90)]{Arnold66}), for eigenfunctions $u\in P_nH$, $v\in P_mH$, and $w\in P_lH$ we obtain 
\begin{align*}
 \left| \left( B(u+v), w \right)_H \right|  &=
 \left|\left( \Pi \left(  \Delta{\psi_v} \nabla \psi_u   \right),   \nabla^{\perp} \psi_w \right)_H 
 +   \left( \Pi \left(  \Delta{\psi_u} \nabla \psi_v   \right),   \nabla^{\perp} \psi_w\right)_H\right|   \\
&= \left| \left( \Pi \left(  \sum_{|j| \le m}{ \psi_v^j \Delta{Y^j_m}}  \sum_{|i| \le n}{ \psi_u^i \nabla Y^i_n  }    \right),  \sum_{|k|\le l}{ \psi_w^l \nabla^{\perp} Y_l^k }   \right)_H  \right.   \\
 &  \left.\qquad + \left( \Pi \left(  \sum_{|j| \le m}{ \psi_v^j \Delta{Y^j_m}}  \sum_{|i| \le n}{ \psi_u^i \nabla Y^i_n  }    \right),  \sum_{|k|\le l}{ \psi_w^l \nabla^{\perp} Y_l^k }   \right)_H      \right|\\
&= \left |  \sum_{|i| \le n}  \sum_{|j| \le m} \sum_{|k| \le l}    \psi_u^i \psi_v^j \psi_w^l \left( \Delta Y^j_m \nabla Y^i_n, \nabla^{\perp} Y^k_l   \right)_H    \right.  \\
&\left.\qquad+  \sum_{|i| \le n}  \sum_{|j| \le m} \sum_{|k| \le l}   \psi_u^i \psi_v^j \psi_w^l  \left( \Delta Y^j_m \nabla Y^i_n, \nabla^{\perp} Y^k_l   \right)_H    \right|
\end{align*}  

An explicit expression for the scalar product $\left( \Delta Y^j_m \nabla Y^i_n, \nabla^{\perp} Y^k_l   \right)_H $ is given
in~\cite[Theorem 5.3]{FenglerFreeden05}, that expression involves the so-called Wigner-3j
symbols. For this symbols we refer also to
~\cite[Section~3.7]{Edmonds96} and~\cite[Section~2]{RaschYu04}.
From that expession in~\cite[Theorem 5.3]{FenglerFreeden05}, recalling that the
Wigner-3j symbol $\left ( \begin{array}{ccc}  j_1 & j_2 &j_3 \\ m_1 & m_2 & m_3   \end{array} \right)$ vanishes unless all the conditions
\begin{enumerate}
\item[i.] $m_1 + m_2 + m_3 = 0$,
\item[ii.] $j_1 + j_2 + j_3$ is an integer (if $m_1 = m_2 = m_3 = 0$, $j_1 + j_2 + j_3$ is an even integer), 
\item[iii.] $|m_k| \le j_k$, and
\item[iv.] $|j_1 - j_2| \le j_3 \le j_1 + j_2$
\end{enumerate}  
are satisfied, we can conclude that $(u,v,w) \in P_nH \times P_mH \times P_lH$ and
$\left(B(u+v), w\right)_H \ne 0$ only if $m -n \le l < m+n$ and $m+n+l$ is odd
(cf.~\cite[Corollary~5.4]{FenglerFreeden05}).

Therefore, we obtain that necessarily ${\rm card}(\FF_{n,\,m}^\bullet)\le 2n$ and ${\rm card}(\FF_{n,\,\bullet}^l)\le 2n$, that is,
Assumption~\ref{A:Fnml} holds for $C_{\FF} = 2$ and $ \zeta  = \frac{1}{2}$.

For $(u,v,w) \in P_nH \times P_mH \times P_lH$ and $\left(B(u+v), w\right)_H \ne 0$ we have $l \in \left[ m-n, m+n \right]$, then
$\lambda_{l} < \lambda_{n+m}$, and from Lemma~\ref{L:triangleK} and~\eqref{polyK} in the Appendix (setting~$p(x)=x(x+1)$), we have that
$\lambda_{m+n}^{\frac{1}{2}} \le 
\lambda_{m}^{\frac{1}{2}} + \lambda_{n}^{\frac{1}{2}} +2$, and it follows that Assumption~\ref{A:Bnml} holds with
$\alpha=\frac{1}{2}$.

\noindent
{\bf The parameters~$\theta$ and~$\xi$.}
From~\cite[Section~3, Lemma~2]{CaoRammahaTiti99}, we can take $\theta=\frac{1}{4}$ in~\ref{R:linf_bound}, and
from~$\lambda_k = k(k+1) > k^2$ it follows that~\eqref{lambda_k} holds with~$\xi=2$.  

\noindent
{\bf Conclusion.}
Taking into account Remark~\ref{R:linf_bound}, we conclude that Theorem~\ref{T:main} holds
with~$\alpha=\frac{1}{2}$ and~$s>\frac{1}{2}$. This agrees with the results in~\cite{CaoRammahaTiti99}.


\subsection{Hemisphere} \label{Hesmisphere} 
Let $\Omega$ be the Hemisphere $\Sp^2_+ = \left \{ (x_1, x_2, x_3) \in \Sp^2 \mid x_3 >0 \right\}$.
On the boundary $\p\Sp^2_+$ of $\Sp^2_+$, $\p\Sp^2_+ = \left \{ (x_1, x_2, x_3) \in \Sp^2 \mid x_3 =0 \right\}$, we impose the Lions boundary conditions, that is,
we consider the evolutionary Navier--Stokes equation
in~$H \coloneqq \{u \in L^2 (\Omega,\,T \Omega) \mid \nabla \cdot u =0\mbox{ and }g(u,\,\nnn)=0\}
\cap \{\nabla^\bot\psi\mid\psi\in H^1(\Sp^2_+,\,\R)\}$, with
$V \coloneqq H\cap H^1 (\Omega,\,T \Omega)$ and
$\D(A) \coloneqq V\cap\{u\in H^2 (\Omega,\,T \Omega)\mid\nabla^\bot\cdot u=0\mbox{ on }\p\Sp^2_+\}$
(cf.~\cite[Sections~5.5 and~6.4]{Rod-Thesis08}).

In this case we will use Theorem~\ref{T:main} and Remark~\ref{R:linf_bound}, and show that as in the case of the Sphere, in
Section~\ref{sS:Sphere},
there we can take $\theta = \frac{1}{4}$,
$\xi = 2$, $\zeta = \frac{1}{2}$, $\alpha=\frac{1}{2}$
and $s > \frac{1}{2}$.

In spherical coordinates~$\Sp^2\sim(\vartheta, \phi)\in[0,\,\pi]\times[0,\,2\pi)$
the Hemisphere corresponds to $\Sp^2_+\sim(\vartheta, \phi)\in[0,\,\frac{\pi}{2})\times[0,\,2\pi)$.
It turns out that from the system~\eqref{completesys-Sp} we can construct a complete system in $H$ formed by eigenfunctions of $A$, it is
\begin{equation}\label{completesys-hemiSp}
\left\{Z^m_n (\vartheta, \phi)\rest{\vartheta\in[0,\,\frac{\pi}{2})}
=\lambda_n^{-\frac{1}{2}} \nabla^{\perp} Y^m_n (\vartheta, \phi)\rest{\vartheta\in[0,\,\frac{\pi}{2})}\left|
\begin{array}{l}n\in\N_0,\; m\in\Z,\\|m|\le n,\;|m|+n\mbox{ is odd}\end{array}\right.\right\}
\end{equation}
(cf.~\cite[Proposition~6.4.2]{Rod-Thesis08}). Let us show that the system is complete.
For~$Z^m_n (\vartheta, \phi)$ in~\eqref{completesys-Sp} we have that
$\nabla^\perp\cdot Z^m_n (\vartheta, \phi)
=\lambda_n^{-\frac{1}{2}} \Delta Y^m_n (\vartheta, \phi)=\lambda_n^{\frac{1}{2}} Y^m_n (\vartheta, \phi)$, and if
$|m|+n$ is odd we have that $Y^m_n (\frac{\pi}{2}, \phi)=0$, that is, $Z^m_n (\vartheta, \phi)\rest{\vartheta\in[0,\,\frac{\pi}{2})}$
is in~$\D(A)$. Further we have that for $\vartheta_1\in[0,\,\frac{\pi}{2}]$,
\[
\begin{array}{rcll}
Z^m_n (\frac{\pi}{2}-\vartheta_1,\, \phi)&=&Z^m_n (\frac{\pi}{2}+\vartheta_1,\, \phi),&\quad\mbox{if }|m|+n\mbox{ is odd};\\
Z^m_n (\frac{\pi}{2}-\vartheta_1,\, \phi)&=&-Z^m_n (\frac{\pi}{2}+\vartheta_1,\, \phi),&\quad\mbox{if }|m|+n\mbox{ is even}.
\end{array}
\]
Notice that from~\eqref{assocLegfunt}, we can see that $P^m_n(-x)=-P^m_n(x)$ if $|m|+n$ is odd, and $P^m_n(-x)=P^m_n(x)$ if $|m|+n$ is even.

To show that~\eqref{completesys-hemiSp} is complete in~$H$, it is sufficient to show that the family of stream functions
$\{Y^m_n (\vartheta, \phi)\rest{\vartheta\in[0,\,\frac{\pi}{2})}
\mid n\in\N_0,\; m\in\Z,\;|m|\le n,\mbox{ and }|m|+n\mbox{ is odd}\}$ form a complete system in~$L^2(\Sp^2_+,\,\R)$.
Let $g(\vartheta,\,\phi)$ be a function defined on the Hemisphere $[0,\,\frac{\pi}{2})\times[0,\,2\pi)$;
we extend it to a function $\tilde{g}$ defined on the Sphere
as follows
\begin{align*}
\tilde{g}(\vartheta,\,\phi) = \left \{\begin{array}{lll}
g(\vartheta,\,\phi) & \text{if} &\vartheta\in[0,\,\frac{\pi}{2}), \\ 
-g(\pi-\vartheta,\,\phi) & \text{if} &\vartheta\in(\frac{\pi}{2},\,\pi].
\end{array} \right.
\end{align*}
We know that we can write
$\tilde{g} = \sum_{(n,m) \in \mathcal{S}}{(\tilde{g}, Y^m_n)_{L^2(\Sp^2,\,\R)} { Y^m_n}}$ where
$\mathcal{S}\coloneqq\{(n,\,m)\in\Z^2\mid n\in\N\mbox{ and }|m|\le n\}$.

By using spherical coordinates, we find for even $|m|+n$ 
\begin{align*}
&\quad\int_{\frac{\pi}{2}}^{\pi} \tilde g(\vartheta, \phi)  Y^m_n(\vartheta,\phi) \sin (\vartheta) \ed\vartheta 
=\int_{0}^{\frac{\pi}{2}} \tilde g(\textstyle\frac{\pi}{2}+\vartheta_1, \phi)
Y^m_n(\textstyle\frac{\pi}{2}+\vartheta_1,\phi) \sin(\frac{\pi}{2}+ \vartheta_1) \ed\vartheta_1 \\
&=\int_{0}^{\frac{\pi}{2}} -g(\textstyle\frac{\pi}{2}-\vartheta_1, \phi) 
Y^m_n(\textstyle\frac{\pi}{2}-\vartheta_1,\phi) \sin(\textstyle\frac{\pi}{2}-\vartheta_1) \ed\vartheta_1\\ 
&=-\int_{\frac{\pi}{2}}^{0} g(\vartheta_2, \phi) 
Y^m_n(\vartheta_2,\phi) \sin(\vartheta_2) (-\ed\vartheta_2) 
=-\int_{0}^{\frac{\pi}{2}} g(\vartheta_2, \phi)  Y^m_n(\vartheta_2,\phi) \sin (\vartheta_2) \ed\vartheta_2,
\end{align*}
which implies $\int_{0}^{\pi} g(\vartheta, \phi)  Y^m_n(\vartheta,\phi) \sin (\vartheta) \ed\vartheta =0$. Hence, for even $|m|+n$,
it follows $(\tilde g, Y_n^m)_{L^2(\Sp^2,\,\R)} 
= \int_{0}^{2\pi}d\phi{\int_{0}^{\pi}{\tilde g(\vartheta, \phi)  Y^m_n(\vartheta,\phi) \sin (\vartheta) \ed\vartheta }}=0$, that is,
$\tilde{g} = \sum_{(n,m) \in \sS_+}{(\tilde{g}, Y^m_n)_{L^2(\Sp^2,\,\R)} { Y^m_n}}$,
with $\sS_+\coloneqq\{(n,m) \in \mathcal{S}\mid|m|+n\mbox{ is odd}\}$, and
\[
g = \tilde{g}\rest{\vartheta\in[0,\,\frac{\pi}{2})}
=\sum_{(n,m) \in \sS_+} (\tilde{g}, Y^m_n)_{L^2(\Sp^2,\,\R)} Y^m_n\rest{\vartheta\in[0,\,\frac{\pi}{2})}, 
\]
which shows that the
set~$\{Y^m_n (\vartheta, \phi)\rest{\vartheta\in[0,\,\frac{\pi}{2})}\mid (n,\,m)\in\sS\}$ is complete in~$L^2(\Sp^2_+,\,\R)$.

Procceding as above for the extension~$\tilde g$ and for odd $|m|+n$ we have 
$(\tilde g, Y_n^m)_{L^2(\Sp^2,\,\R)}=2(g, Y_n^m\rest{\vartheta\in[0,\,\frac{\pi}{2})})_{L^2(\Sp^2_+,\,\R)}$, and also
$Y_n^m=\tilde h$ with $h=Y_n^m\rest{\vartheta\in[0,\,\frac{\pi}{2})}$. In particular we conclude that the family
$\{Y^m_n (\vartheta, \phi)\rest{\vartheta\in[0,\,\frac{\pi}{2})}\mid (n,\,m)\in\sS\}$ is orthogonal in~$L^2(\Sp^2_+,\,\R)$ and then it forms a
basis in~$L^2(\Sp^2_+,\,\R)$.

As a consequence we can conclude that the family~\eqref{completesys-hemiSp} form a complete system in~$H$. Notice that for $n=0$, $Y_0^0$
is a constant function, and the vector field~$\nabla^\perp Y_0^0\in L^2(\Sp^2,\,T\Sp^2)$ vanishes. From the fact that
$(Y_n^m, Y_n^m)_{L^2(\Sp^2,\,\R)}=2(Y_n^m\rest{\vartheta\in[0,\,\frac{\pi}{2})}, Y_n^m\rest{\vartheta\in[0,\,\frac{\pi}{2})})_{L^2(\Sp^2_+,\,\R)}$,
we can normalize that system as
\begin{equation}\label{completesys-hemiSp-n}
\left\{\sqrt2 Z^m_n (\vartheta, \phi)\rest{\vartheta\in[0,\,\frac{\pi}{2})}
=\lambda_n^{-\frac{1}{2}} \nabla^{\perp} Y^m_n (\vartheta, \phi)\rest{\vartheta\in[0,\,\frac{\pi}{2})}\left|
\begin{array}{l}n\in\N_0,\; m\in\Z,\\|m|\le n,\;|m|+n\mbox{ is odd}\end{array}\right.\right\}
\end{equation}

%

\noindent
{\bf Conclusion.} We can follow the arguments in the case of the Sphere, in Section~\ref{sS:Sphere},
to conclude that Theorem~\ref{T:main} holds with~$\alpha=\frac{1}{2}$ and~$s>\frac{1}{2}$.


\subsection{Rectangle} \label{sS:Rectangle}
Let $\Omega$ be the two-dimensional Rectangle $\Omega = (0, a) \times (0,b)\subset\R^2$.
On the boundary we impose the Lions boundary conditions, that is,
we consider the evolutionary Navier--Stokes equation
in~$H \coloneqq \{u \in L^2 (\Omega,\,\R^2) \mid \nabla \cdot u =0\mbox{ and } u\cdot\nnn=0\mbox{ on }\p\Omega\}$, with
$V \coloneqq H\cap H^1 (\Omega,\,\R^2)$ and
$\D(A) \coloneqq V\cap\{u\in H^2 (\Omega,\,\R^2)\mid\nabla^\bot\cdot u=0\mbox{ on }\p\Omega\}$
(cf.~\cite{Rod06} and~\cite[Section~6.3]{Rod-Thesis08}).

We will show that in this case we can take $\alpha=\frac{1}{2}$,
$\xi=\frac{1}{2}$, and $\zeta=0$ in Theorem~\ref{T:main_r}, and $\theta=0$ in Remark~\ref{R:linf_bound}. That is, we can take~$s>0$,
in Theorem~\ref{T:main_r}.

The complete system of eigenfunctions $\{Y_{(k_1,\,k_2)}\mid (k_1,\,k_2)\in\N_0^2\}$ and respective eigenvalues
$\{\lambda_{(k_1,\,k_2)}\mid (k_1,\,k_2)\in\N_0^2\}$ of $A$, can be found in~\cite[Sections~2.2 and 2.3]{Rod06}, they are given by
\begin{equation}\label{compsys-rect}
Y_{(k_1,k_2)}\coloneqq\left({\begin{array}{r}
-\frac{k_2 \pi}{b} \sin \left( \frac{k_1 \pi x_1}{a} \right) \cos \left( \frac{k_2 \pi x_2}{b}  \right)\\
\frac{k_1 \pi}{a} \cos \left( \frac{k_1 \pi x_1}{a} \right) \sin \left( \frac{k_2 \pi x_2}{b}  \right)
\end{array} }
\right),\quad \lambda_{(k_1,\,k_2)}\coloneqq\pi^2 \left( \textstyle\frac{k_1^2}{a^2} + \textstyle\frac{k_2^2}{b^2} \right).
\end{equation}

Though, the above systems are indexed over~$\N_0^2$, like in Section~\ref{sS:Torus},
we can check the Assumptions~\ref{A:A}, \ref{A:B},
\ref{A:Bnml_r}, and~\ref{A:Fnml_r} without rewriting the families as indexed over $\N_0$.

We may normalize the family~\eqref{compsys-rect}, obtaining the system
$\{ W_{(k_1,k_2)} \mid (k_1,\,k_2)\in\N_0^2\}$, with
\begin{equation}\label{compsys-rect-n}
W_{(k_1,k_2)}\coloneqq2(ab\lambda_{(k_1,\,k_2)})^{-\frac{1}{2}}\left({\begin{array}{r}
-\frac{k_2 \pi}{b} \sin \left( \frac{k_1 \pi x_1}{a} \right) \cos \left( \frac{k_2 \pi x_2}{b}  \right)\\
\frac{k_1 \pi}{a} \cos \left( \frac{k_1 \pi x_1}{a} \right) \sin \left( \frac{k_2 \pi x_2}{b}  \right)
\end{array} }
\right).
\end{equation}

\noindent
{\bf Checking Assumptions~\ref{A:A} and~\ref{A:B}.} We need only to check that $B(u)=0$ if $u$ is an eigenfuntion; this follows
from~\cite[equation~(6.1)]{Rod06}. For the other points we refer
to~\cite{Rod06} and~\cite[Section~2.3]{Temam95}.

\noindent
{\bf Checking Assumptions~\ref{A:Bnml_r} and~\ref{A:Fnml_r}.} From~\cite[equation~(6.1)]{Rod06} (cf.~\cite[equation~(6.4)]{Rod-Thesis08}),
and $B(Y_{(n_1,\,n_2)} + Y_{(m_1,\,m_2)})=B(Y_{(n_1,\,n_2)},\,Y_{(m_1,\,m_2)})+B(Y_{(m_1,\,m_2)},\,Y_{(n_1,\,n_2)})$, we can derive that
$\left(B(Y_{(n_1,\,n_2)} + Y_{(m_1,\,m_2)}),\, Y_{(l_1,\,l_2)}\right)_H\ne0$ only if
\begin{align*}
l_1 = |n_1\pm m_1|\mbox{ and }l_2=|n_2\pm m_2|,
\end{align*}
which implies that ${\rm card}(\underline\FF_{n,\,m}^\bullet) \le 4$ and ${\rm card}(\underline\FF_{n,\,\bullet}^l) \le {4}$. That is,
Assumption~\ref{A:Fnml_r} holds with $\zeta=0$. We also see that necessarily $\lambda_{(l_1,\,l_2)}
\le\lambda_{(n_1+ m_1,\,n_2+ m_2)}$; noticing that $(k_1,\,k_2)\mapsto\lambda_{(k_1,\,k_2)}$ is a scalar product, or using
Lemma~\ref{L:triangleK}, we conclude that $\lambda_{(l_1,\,l_2)}^\frac{1}{2}
\le\lambda_{(n_1,\,n_2)}^\frac{1}{2}+\lambda_{(m_1,\,m_2)}^\frac{1}{2}$, that is, Assumption~\ref{A:Bnml_r} holds with $\alpha=\frac{1}{2}$.

\noindent
{\bf Looking for the value~$\theta$ in Remark~\ref{R:linf_bound}.}
We have that
\begin{align*}
|W_{(k_1,\,k_2)}|_{L^{\infty}(\Omega,\,\R^2)}^2
&=\max_{(x_1,\,x_2)\in\Omega}|W_{(k_1,\,k_2)}(x_1,\,x_2)|_{\R^2}^2\le4(ab\lambda_{(k_1,\,k_2)})^{-1}(\textstyle\frac{k_2^2 \pi^2}{b^2}+
\textstyle\frac{k_1^2 \pi^2}{a^2})\\
&=4(ab)^{-1},
\end{align*}
that is, we can take $\theta=0$.
%
%

\noindent
{\bf Asymptotic behavior of the (repeated) eigenvalues.}
We recall that for an open domain $\Omega\subset\R^2$, under Lions boundary conditions, the eigenvalues
of the Stokes operator~$A:\D(A)\to H$ are
those of the Dirichlet Laplacian $\Delta H^2(\Omega,\,\R)\cap H^1_0(\Omega,\,\R)\to L^2(\Omega,\,\R)$, that is,
$Au=\lambda u$ if, and only if, $\Delta \nabla^\perp u=\lambda\nabla^\perp u$. Thus, from~\cite[Corollary~1]{LiYau83} we have that
we can take $\rho<\frac{2\pi}{ab}$ and $\xi=1$ in~Theorem~\ref{T:main_r}.

For the sake of completeness we would like also to refer to the results in~\cite{Ilyin09}, and references therein,
for the case of no-slip boundary conditions.

\noindent
{\bf Conclusion.}
Taking into account Remark~\ref{R:linf_bound}, we conclude that Theorem~\ref{T:main_r} holds with~$\alpha=\frac{1}{2}$ and~$s>0$.

\subsection{Cylinder} \label{Cylinder}
Let $\Omega$ be a two-dimensional Cylinder $\Omega = \left( \frac{a}{2 \pi} \Sp^1 \right) \times (0,\,b)\sim (0,\,a)\times (0,\,b)$.
On the boundary~$(0,\,a)\times \{0,\,b\}$ we impose the Lions boundary conditions, that is,
we consider the evolutionary Navier--Stokes equation
in~$H \coloneqq \{u \in L^2 (\Omega,\,\R^2) \mid \nabla \cdot u =0\mbox{ and } u\cdot\nnn=0\mbox{ on }\p\Omega\}$, with
$V \coloneqq H\cap H^1 (\Omega,\,\R^2)$ and
$\D(A) \coloneqq V\cap\{u\in H^2 (\Omega,\,\R^2)\mid\nabla^\bot\cdot u=0\mbox{ on }\p\Omega\}$. We can see the domain~$\Omega$ as an infinite
channel $\R\times (0,\,b)$ where we take $a$-periodic boundary conditions on the infinite direction $x_1\in\R$ and Lions boundary conditions
on the boundary $\R\times\{0,\,b\}$. 

We will show that in this case we can take $\alpha=\frac{1}{2}$,
$\xi=\frac{1}{2}$, and $\zeta=0$ in Theorem~\ref{T:main_r}, and $\theta=0$ in Remark~\ref{R:linf_bound}. That is, we can take~$s>0$,
in Theorem~\ref{T:main_r}.

A complete system of orthogonal eigenfunctions of $A$
$\{ Y_n^{\varsigma}, Y_m^{\varkappa} \big|~n \in \N_0^2, m \in \N \times \N_0 \}$, and corresponding eigenvalues
$\{ \lambda_n^\varsigma, \lambda_m^\varkappa \big|~n \in \N_0^2, m \in \N \times \N_0 \}$, are  given by
\begin{equation}\label{compsys-cyli}
 \begin{array}{rcr}
Y_k^\varsigma = Y_{(k_1, k_2)}^\varsigma &=&  \left({\begin{array}{r}
-\frac{k_2 \pi}{b} \sin \left( \frac{2k_1 \pi x_1}{a} \right) \cos \left( \frac{k_2 \pi x_2}{b}  \right)\\
\frac{2k_1 \pi}{a} \cos \left( \frac{2k_1 \pi x_1}{a} \right) \sin \left( \frac{k_2 \pi x_2}{b}  \right)
\end{array} }
\right),
\\
Y_k^\varkappa = Y_{(k_1, k_2)}^\varkappa &=&  \left({\begin{array}{r}
-\frac{k_2 \pi}{b} \cos \left( \frac{2k_1 \pi x_1}{a} \right) \cos \left( \frac{k_2 \pi x_2}{b}  \right)\\
-\frac{2k_1 \pi}{a} \sin \left( \frac{2k_1 \pi x_1}{a} \right) \sin \left( \frac{k_2 \pi x_2}{b}  \right)
\end{array} }
\right),
\end{array}
\end{equation}
and $\lambda_{(k_1, k_2)}^\varsigma=\lambda_{(k_1, k_2)}^\varkappa=\lambda_{(k_1, k_2)}\coloneqq\pi^2 \left( \frac{(2k_1)^2}{a^2}
+ \frac{k_2^2}{b^2} \right)$.
\begin{remark}
Notice that $Y_n^{\varsigma}=\nabla^\perp\psi_n^\varsigma$,
$Y_m^{\varkappa}=\nabla^\perp\psi_n^\varkappa$, with
$\psi_n^\varsigma\coloneqq\sin \left( \frac{2k_1 \pi x_1}{a} \right) \sin \left( \frac{k_2 \pi x_2}{b}\right)$ and
$\psi_n^\varkappa\coloneqq\cos \left( \frac{2k_1 \pi x_1}{a} \right) \sin \left( \frac{k_2 \pi x_2}{b}\right)$; notice also that
the set of stream functions
$\{ Y_n^{\varsigma}, Y_m^{\varkappa} \big|~n \in \N_0^2, m \in \N \times \N_0 \},\big|~n \in \N_0^2, m \in \N \times \N_0   \}$ is
an orthogonal and complete, in $L^2(\Omega,\,\R^2)$, system of eigenfunctions of the Laplacian in~$\Omega\sim(0,\,a)\times (0,\,b)$.
\end{remark}

We may normalize the family, obtaining the normalized system $\{ W_n^\varsigma, W_m^\varkappa \big|~n \in \N_0^2, m \in \N \times \N_0 \}$, with
\begin{equation}\label{compsys-cyli-n}
W_k^\varsigma \coloneqq  2(ab\lambda_k)^{-\frac{1}{2}}Y_k^\varsigma,\quad
W_k^\varkappa \coloneqq 2(ab\lambda_k)^{-\frac{1}{2}}Y_k^\varkappa.
\end{equation}

Now we check our assumptions, proceeding as in the case of the Rectangle, in Section~\ref{sS:Rectangle}.

\noindent
{\bf Checking Assumptions~\ref{A:A} and~\ref{A:B}.} The assumptions follow
by reasoning as in the case of the Sphere in Section~\ref{sS:Sphere}, where now $g(\cdot,\,\cdot)=(\cdot,\,\cdot)_{\R^2}$ is the usual Euclidean
scalar product in~$\R^2$.


\noindent
{\bf Checking Assumptions~\ref{A:Bnml_r} and~\ref{A:Fnml_r}.} From the discussion following Corollary~5.6.3 in~\cite{Rod-Thesis08}
we can conclude that $\nabla^\perp\cdot (B(u,\,v)+ B(u,\,v))
=(\nabla\Delta^{-1}\nabla^\perp\cdot v,\,\nabla^\perp\nabla^\perp\cdot u)_{\R^2}
+(\nabla\Delta^{-1}\nabla^\perp\cdot u,\,\nabla^\perp\nabla^\perp\cdot v)_{\R^2}$. If $u$ and $v$ are eigenfunctions
from~\eqref{compsys-cyli} with
associated eigenvalues~$\lambda_u$ and~$\lambda_v$, and associated eigenfunctions~$\psi_u$ and~$\psi_v$, we obtain
\begin{align}
\nabla^\perp\cdot (B(u,\,v)+ B(u,\,v))
&=\lambda_u(\nabla \psi_v,\,\nabla^\perp \psi_u)_{\R^2}
+\lambda_v(\nabla \psi_u,\,\nabla^\perp \psi_v)_{\R^2}\notag\\
&=(\lambda_u-\lambda_v)(\nabla^\perp \psi_u,\,\nabla \psi_v)_{\R^2}\label{vortBuv}.
\end{align}
From straightforward computations, we find the following expressions
\begin{align*}
\begin{array}{rl}
 &(\nabla^\perp \psi_n^\varsigma,\,\nabla \psi_m^\varsigma)_{\R^2}\\
=&\left(\begin{array}{r}
-\frac{n_2 \pi}{b} \sin \left( \frac{2n_1 \pi x_1}{a} \right) \cos \left( \frac{n_2 \pi x_2}{b}  \right)\\
\frac{2n_1 \pi}{a} \cos \left( \frac{2n_1 \pi x_1}{a} \right) \sin \left( \frac{n_2 \pi x_2}{b}  \right)
\end{array}\right)\cdot
\left(\begin{array}{r}
\frac{2m_1 \pi}{a} \cos \left( \frac{2m_1 \pi x_1}{a} \right) \sin \left( \frac{m_2 \pi x_2}{b}  \right)\\
\frac{m_2 \pi}{b} \sin \left( \frac{2m_1 \pi x_1}{a} \right) \cos \left( \frac{m_2 \pi x_2}{b}  \right)
\end{array}\right)\\
=&-\frac{2 \pi^2n_2m_1}{ab} \sin \left( \frac{2n_1 \pi x_1}{a} \right) \cos \left( \frac{n_2 \pi x_2}{b}\right) 
\cos \left( \frac{2m_1 \pi x_1}{a} \right) \sin \left( \frac{m_2 \pi x_2}{b}\right)\\
&+\frac{2 \pi^2n_1m_2}{ab}  \cos \left( \frac{2n_1 \pi x_1}{a} \right) \sin \left( \frac{n_2 \pi x_2}{b}\right)
\sin \left( \frac{2m_1 \pi x_1}{a} \right) \cos \left( \frac{m_2 \pi x_2}{b}\right);
\end{array}
\end{align*}

\begin{equation*}
\begin{array}{rl}
 &(\nabla^\perp \psi_n^\varsigma,\,\nabla \psi_m^\varkappa)_{\R^2}\\
=&\left(\begin{array}{r}
-\frac{n_2 \pi}{b} \sin \left( \frac{2n_1 \pi x_1}{a} \right) \cos \left( \frac{n_2 \pi x_2}{b}  \right)\\
\frac{2n_1 \pi}{a} \cos \left( \frac{2n_1 \pi x_1}{a} \right) \sin \left( \frac{n_2 \pi x_2}{b}  \right)
\end{array}\right)\cdot
\left(\begin{array}{r}
-\frac{2m_1 \pi}{a} \sin \left( \frac{2m_1 \pi x_1}{a} \right) \sin \left( \frac{m_2 \pi x_2}{b}  \right)\\
\frac{m_2 \pi}{b} \cos \left( \frac{2m_1 \pi x_1}{a} \right) \cos \left( \frac{m_2 \pi x_2}{b}  \right)
\end{array}\right)\\
=&\frac{2 \pi^2n_2m_1}{ab} \sin \left( \frac{2n_1 \pi x_1}{a} \right) \cos \left( \frac{n_2 \pi x_2}{b}\right) 
\sin \left( \frac{2m_1 \pi x_1}{a} \right) \sin \left( \frac{m_2 \pi x_2}{b}\right)\\
&+\frac{2 \pi^2n_1m_2}{ab}  \cos \left( \frac{2n_1 \pi x_1}{a} \right) \sin \left( \frac{n_2 \pi x_2}{b}\right)
\cos \left( \frac{2m_1 \pi x_1}{a} \right) \cos \left( \frac{m_2 \pi x_2}{b}\right);
\end{array}
\end{equation*}

\begin{equation*}
\begin{array}{rl}
 &(\nabla^\perp \psi_n^\varkappa,\,\nabla \psi_m^\varkappa)_{\R^2}\\
=&\left(\begin{array}{r}
-\frac{n_2 \pi}{b} \cos \left( \frac{2n_1 \pi x_1}{a} \right) \cos \left( \frac{n_2 \pi x_2}{b}  \right)\\
-\frac{2n_1 \pi}{a} \sin \left( \frac{2n_1 \pi x_1}{a} \right) \sin \left( \frac{n_2 \pi x_2}{b}  \right)
\end{array}\right)\cdot
\left(\begin{array}{r}
-\frac{2m_1 \pi}{a} \sin \left( \frac{2m_1 \pi x_1}{a} \right) \sin \left( \frac{m_2 \pi x_2}{b}  \right)\\
\frac{m_2 \pi}{b} \cos \left( \frac{2m_1 \pi x_1}{a} \right) \cos \left( \frac{m_2 \pi x_2}{b}  \right)
\end{array}\right)\\
=&\frac{2 \pi^2n_2m_1}{ab} \cos \left( \frac{2n_1 \pi x_1}{a} \right) \cos \left( \frac{n_2 \pi x_2}{b}\right) 
\sin \left( \frac{2m_1 \pi x_1}{a} \right) \sin \left( \frac{m_2 \pi x_2}{b}\right)\\
&-\frac{2 \pi^2n_1m_2}{ab}  \sin \left( \frac{2n_1 \pi x_1}{a} \right) \sin \left( \frac{n_2 \pi x_2}{b}\right)
\cos \left( \frac{2m_1 \pi x_1}{a} \right) \cos \left( \frac{m_2 \pi x_2}{b}\right).
\end{array}
\end{equation*}

Thus, if we denote
$\varsigma_c^{z,\,l}\coloneqq\sin \left( \frac{l \pi z}{c} \right)$ and $\varkappa_c^{z,\,l}\coloneqq\cos \left( \frac{l \pi z}{c} \right)$; $m \wedge n = m_1 n_2 - m_2 n_1$ and $m \vee n = m_1 n_2 + m_2 n_1$,
we obtain
\begin{align}\label{vorteigs-aa}
\begin{array}{rl}
 &(\nabla^\perp \psi_n^\varsigma,\,\nabla \psi_m^\varsigma)_{\R^2}\\
=&-\frac{\pi^2n_2m_1}{2ab} \left(\varsigma_a^{x_1,\,2(n_1+m_1)}+\varsigma_a^{x_1,\,2(n_1-m_1)}\right)
\left(\varsigma_b^{x_2,\,n_2+m_2}-\varsigma_b^{x_2,\,n_2-m_2}\right)\\
&+\frac{\pi^2n_1m_2}{2ab}  \left(\varsigma_a^{x_1,\,2(n_1+m_1)}-\varsigma_a^{x_1,\,2(n_1-m_1)}\right)
\left(\varsigma_b^{x_2,\,m_2+n_2}+\varsigma_b^{x_2,\,n_2-m_2}\right)\\
=&-\frac{\pi^2m\wedge n}{2ab}\varsigma_a^{x_1,\,2(n_1+m_1)}\varsigma_b^{x_2,\,n_2+m_2}
+\frac{\pi^2m\vee n}{2ab}\varsigma_a^{x_1,\,2(n_1+m_1)}\varsigma_b^{x_2,\,n_2-m_2}\\
&-\frac{\pi^2m\vee n}{2ab}\varsigma_a^{x_1,\,2(n_1-m_1)}\varsigma_b^{x_2,\,n_2+m_2}
+\frac{\pi^2m\wedge n}{2ab}\varsigma_a^{x_1,\,2(n_1-m_1)}\varsigma_b^{x_2,\,n_2-m_2}.
\end{array}
\end{align}

\begin{equation}\label{vorteigs-ab}
\begin{array}{rl}
 &(\nabla \psi_n^\varsigma,\,\nabla^\perp \psi_m^\varkappa)_{\R^2}\\
=&+\frac{\pi^2n_2m_1}{2ab} \left(\varkappa_a^{x_1,\,2(n_1-m_1)}-\varkappa_a^{x_1,\,2(n_1+m_1)}\right)
\left(\varsigma_b^{x_2,\,n_2+m_2}-\varsigma_b^{x_2,\,n_2-,m_2}\right)\\
&+\frac{\pi^2n_1m_2}{2ab}  \left(\varkappa_a^{x_1,\,2(n_1+m_1)}+\varkappa_a^{x_1,\,2(n_1-m_1)}\right)
\left(\varsigma_b^{x_2,\,n_2+m_2}+\varsigma_b^{x_2,\,n_2-m_2}\right)\\
=&-\frac{\pi^2m\wedge n}{2ab}\varkappa_a^{x_1,\,2(n_1+m_1)}\varsigma_b^{x_2,\,n_2+m_2}
+\frac{\pi^2m\vee n}{2ab}\varkappa_a^{x_1,\,2(n_1+m_1)}\varsigma_b^{x_2,\,n_2-m_2}\\
&+\frac{\pi^2m\vee n}{2ab}\varkappa_a^{x_1,\,2(n_1-m_1)}\varsigma_b^{x_2,\,n_2+m_2}
-\frac{\pi^2m\wedge n}{2ab}\varkappa_a^{x_1,\,2(n_1-m_1)}\varsigma_b^{x_2,\,n_2-m_2};
\end{array}
\end{equation}

\begin{equation}\label{vorteigs-bb}
\begin{array}{rl}
 &(\nabla \psi_n^\varkappa,\,\nabla^\perp \psi_m^\varkappa)_{\R^2}\\
=&\frac{\pi^2n_2m_1}{2ab} \left(\varsigma_a^{x_1,\,2(n_1+m_1)}-\varsigma_a^{x_1,\,2(n_1-m_1)}\right)
\left(\varsigma_b^{x_2,\,n_2+m_2}-\varsigma_b^{x_2,\,n_2-m_2}\right)\\
&-\frac{\pi^2n_1m_2}{2ab}  \left(\varsigma_a^{x_1,\,2(n_1+m_1)}+\varsigma_a^{x_1,\,2(n_1-m_1)}\right)
\left(\varsigma_b^{x_2,\,m_2+n_2}+\varsigma_b^{x_2,\,n_2-m_2}\right)\\
=&\frac{\pi^2m\wedge n}{2ab}\varsigma_a^{x_1,\,2(n_1+m_1)}\varsigma_b^{x_2,\,n_2+m_2}
-\frac{\pi^2m\vee n}{2ab}\varsigma_a^{x_1,\,2(n_1+m_1)}\varsigma_b^{x_2,\,n_2-m_2}\\
&-\frac{\pi^2m\vee n}{2ab}\varsigma_a^{x_1,\,2(n_1-m_1)}\varsigma_b^{x_2,\,n_2+m_2}
+\frac{\pi^2m\wedge n}{2ab}\varsigma_a^{x_1,\,2(n_1-m_1)}\varsigma_b^{x_2,\,n_2-m_2}.
\end{array}
\end{equation}

Hence from~$u=\nabla^\perp\psi_u=\nabla^\perp\Delta^{-1}\nabla^\perp\cdot u$,
\eqref{vortBuv}, \eqref{vorteigs-aa}, \eqref{vorteigs-ab} and~\eqref{vorteigs-bb}, we obtain
\begin{align}\label{vorteigs3}
\begin{array}{rl}
&B(Y_n^\varsigma,\,Y_m^\varsigma)+ B(Y_m^\varsigma,\,Y_n^\varsigma)\\
=&\frac{\lambda_n-\lambda_m}{\lambda_{m(++)n}}\frac{\pi^2n\wedge m}{2ab}Y_{n(++)m}^\varsigma
+\frac{\lambda_n-\lambda_m}{\lambda_{n(+-)m}}\frac{\pi^2n\vee m}{2ab}{\scriptstyle\sign(n_2-m_2)}Y_{n(+-)m}^\varsigma\\
&-\frac{\lambda_n-\lambda_m}{\lambda_{n(-+)m}}\frac{\pi^2n\vee m}{2ab}{\scriptstyle\sign(n_1-m_1)}Y_{n(-+)m}^\varsigma
-\frac{\lambda_n-\lambda_m}{\lambda_{m(--)n}}\frac{\pi^2n\wedge m}{2ab}{\scriptstyle\sign(n_1-m_1)\sign(n_2-m_2)}Y_{n(--)m}^\varsigma;\\
&\\
&B(Y_n^\varsigma,\,Y_m^\varkappa)+ B(Y_m^\varkappa,\,Y_n^\varsigma)\\
=&\frac{\lambda_n-\lambda_m}{\lambda_{n(++)m}}\frac{\pi^2n\wedge m}{2ab}Y_{n(++)m}^\varkappa
+\frac{\lambda_n-\lambda_m}{\lambda_{n(+-)m}}\frac{\pi^2n\vee m}{2ab}{\scriptstyle\sign(n_2-m_2)}Y_{n(+-)m}^\varkappa\\
&+\frac{\lambda_n-\lambda_m}{\lambda_{n(-+)n}}\frac{\pi^2n\vee m}{2ab}Y_{n(-+)m}^\varkappa
+\frac{\lambda_n-\lambda_m}{\lambda_{n(--)m}}\frac{\pi^2n\wedge m}{2ab}{\scriptstyle\sign(n_2-m_2)}Y_{n(+-)m}^\varkappa;\\
&\\
&B(Y_n^\varkappa,\,Y_m^\varkappa)+ B(Y_m^\varkappa,\,Y_n^\varkappa)\\
=&-\frac{\lambda_n-\lambda_m}{\lambda_{n(++)m}}\frac{\pi^2n\wedge m}{2ab}Y_{n(++)m}^\varsigma
-\frac{\lambda_n-\lambda_m}{\lambda_{n(+-)m}}\frac{\pi^2n\vee m}{2ab}{\scriptstyle\sign(n_2-m_2)}Y_{n(+-)m}^\varsigma\\
&-\frac{\lambda_n-\lambda_m}{\lambda_{n(-+)n}}\frac{\pi^2n\vee m}{2ab}{\scriptstyle\sign(n_1-m_1)}Y_{n(-+)m}^\varsigma
-\frac{\lambda_n-\lambda_m}{\lambda_{n(--)m}}\frac{\pi^2n\wedge m}{2ab}{\scriptstyle\sign(n_1-m_1)\sign(n_2-m_2)}Y_{n(--)m}^\varsigma;
\end{array}
\end{align}
where $n(\star_1\star_2)m\coloneqq(|n_1\star_1 m_1|,\,|n_2\star_2 m_2|)\in\N^2$, with $\{\star_1,\,\star_2\}\in \{-,\,+\}^2$ and
for $k_1\in\N$, $Y_{(k_1,\,0)}^\varsigma\coloneqq Y_{(k_1,\,0)}^\varkappa\coloneqq 0$. Notice that these
are expressions similar to that obtained for the case of the Rectangle
in~\cite[equation~(6.1)]{Rod06}, \cite[equation~(6.4)]{Rod-Thesis08}.
Notice also that in~\cite[Section~2.3]{Rod06} the eigenvalues are
negative and here they are positive, this is because in~\cite{Rod06} it is considered the
usual Laplacian~$\Delta$ in~$(0,\,a)\times(0,\,b)$ and here (cf.~the discussion following Equation~(2))
we consider the Laplace--de~Rham operator $\Delta_{\Omega}=-\Delta=A$.

From~\eqref{vorteigs3}, and~\eqref{compsys-cyli-n}, we conclude that ${\rm card}(\underline\FF_{n,\,m}^\bullet) \le 4$
and ${\rm card}(\underline\FF_{n,\,\bullet}^l) \le {8}$. That is,
Assumption~\ref{A:Fnml_r} holds with $\zeta=0$. We also see that necessarily $\lambda_{(l_1,\,l_2)}
\le\lambda_{(n_1+ m_1,\,n_2+ m_2)}$; and from
Lemma~\ref{L:triangleK} we conclude that $\lambda_{(l_1,\,l_2)}^\frac{1}{2}
\le\lambda_{(n_1,\,n_2)}^\frac{1}{2}+\lambda_{(m_1,\,m_2)}^\frac{1}{2}$, that is, Assumption~\ref{A:Bnml_r} holds with $\alpha=\frac{1}{2}$.

\noindent
{\bf Looking for the value~$\theta$ in Remark~\ref{R:linf_bound}.}
We can take $\theta=0$, because, proceeding as in the case of the Rectangle, in Section~\ref{sS:Rectangle}, we obtain
\begin{equation*}
|W_{k}^\varsigma|_{L^{\infty}(\Omega,\,\R^2)}^2
\le4(ab)^{-1},\quad |W_{k}^\varkappa|_{L^{\infty}(\Omega,\,\R^2)}^2
\le4(ab)^{-1}.
\end{equation*}
%
%

\noindent
{\bf Asymptotic behavior of the (repeated) eigenvalues.}
Notice that the family $\{\lambda_{k}\mid k\in N_0^2\}=\{\lambda_{k}^\varsigma\mid k\in \N_0^2\}
=\{\lambda_{k}^\varkappa\mid k\in \N_0^2\}$ is a subset of $\{\lambda_k^R\mid k\in \N_0^2\}$ where $\lambda_k^R$ are the
eigenvalues of the Dirichlet Laplacian on the Rectangle~$R\coloneqq(0,\,a)\times(0,\,b)$, in Section~\ref{sS:Rectangle}. Hence,
ordering the families as
$\{\lambda_{k}\mid k\in \N_0^2\}=\{\tilde\lambda_n\mid n\in\N_0\}$ and
$\{\lambda_{k}^R\mid k\in \N_0^2\}=\{\tilde\lambda_n^R\mid n\in\N_0\}$ such that $\tilde\lambda_{n}\le\tilde\lambda_{n+1}$
and $\tilde\lambda_{n}^R\le\tilde\lambda_{n+1}^R$, we can conclude that $(\tilde\lambda_{n})_{n\in\N_0}$ is a subsequence of
$(\tilde\lambda_{n}^R)_{n\in\N_0}$. Now we already know that

\[
 \tilde\lambda_{n}^R\ge\frac{2\pi}{ab}n,
\]
which implies $\tilde\lambda_{n}\ge\tilde\lambda_{n}^R\ge\frac{2\pi}{ab}n$
for all $n\in\N_0$. The family $\{\lambda_{k}\mid k\in \N_0^2\}$ is repeated twice
$\lambda_{k}=\lambda_{k}^\varsigma=\lambda_{k}^\varkappa$ 
for $k\in \N_0^2$. Then for the ordered families, we can write
\[
 \tilde\lambda_{n}^\varsigma\ge\tilde\lambda_{n}^R\ge\frac{2\pi}{ab}n^2\mbox{ and }
 \tilde\lambda_{n}^\varkappa\ge\tilde\lambda_{n}^R\ge\frac{2\pi}{ab}n^2\mbox{ for all }n\in\N_0.
\]
Finally the family of eigenvalues $\{\tilde\lambda_{n}^{\varkappa,\,0}\coloneqq\lambda_{(0,\,n)}^\varkappa\mid n\in N_0\}$ satisfies
\[
  \tilde\lambda_{n}^{\varkappa,\,0}=\frac{\pi^2}{b^2}n^2\mbox{ for all }n\in\N_0.
\]
In particular ordering the
set~$\{\tilde\lambda_{n}^\varsigma,\,\tilde\lambda_{n}^\varkappa,\,\tilde\lambda_{n}^{\varkappa,\,0}\mid n\in\N_0\}$,
in a nondecreasing way, we obtain the sequence of repeated eigenvalues $(\underline\lambda_n)_{n\in\N_0}$
in the case of the Cylinder. Moreover setting $\varrho=\min\{\frac{2\pi}{ab},\,\frac{\pi^2}{b^2}\}$ we find
that there are at most $3n$ elements in the set
$\{\underline\lambda_n\mid n\in\N_0\}$ that are not bigger than~$\varrho (n+1)$; which implies that
$\underline\lambda_{3n+1}\ge\varrho (n+1)$. Hence, since $\underline\lambda_{3(n+1)}\ge\underline\lambda_{3n+2}\ge\underline\lambda_{3n+1}$
we can conclude that for $m\ge4$, $\underline\lambda_{m}\ge\varrho\lfloor\frac{m+2}{3}\rfloor$ where for a positive real number $r$,
$\lfloor r \rfloor$ stands for the biggest integer
below~$r$, that is $r\in\N$ and $r=\lfloor r\rfloor+r_1$ with $r_1\in[0,\,1)$. In particular,
from $\lfloor\frac{m+2}{3}\rfloor\ge\frac{m-1}{3}=\frac{m-1}{3m}m$, we find $\underline\lambda_{m}\ge\frac{\varrho}{4} m$ for $m\ge4$.
So for $\varrho_0\coloneqq\min_{k\in\{1,\,2,\,3\}}\{\underline\lambda_{k}\}$ and $\rho_1\coloneqq\min\{\frac{\varrho}{4},\,\varrho_0\}$,
we have that $\underline\lambda_{m}\ge\rho_1 m$ for all $m\in\N_0$.
%
Thus we can take $\rho<\rho_1$ and $\xi=1$ in Theorem~\ref{T:main_r}.

\noindent
{\bf Conclusion.}
Taking into account Remark~\ref{R:linf_bound}, we conclude that Theorem~\ref{T:main_r} holds with~$\alpha=\frac{1}{2}$ and~$s>0$

\medskip
\appendix
\addcontentsline{toc}{section}{Appendix}
\begin{center}
{\sc --- Appendix ---}
\end{center}
\setcounter{section}{1}
\setcounter{theorem}{0} \setcounter{equation}{0}
\numberwithin{equation}{section}

\subsection{Proof of Proposition~\ref{P:ab_s}}\label{sApx:proofP:ab_s}

The inequalities in Proposition~\ref{P:ab_s} are clear for $s=0$ and $s=1$. Now let $s>0$, $s\ne1$, and consider the quotient
$f(x,\,y)\coloneqq\frac{(x+y)^s}{x^s+y^s}$ for $x\ge0$, $y\ge0$, and $(x,\,y)\ne(0,\,0)$. The gradient of $f$ is given by
\begin{align*}
  \nabla f&=\left[\begin{array}{l}
s(x+y)^{s-1}(x^s+y^s)^{-1}-s(x+y)^sx^{s-1}(x^s+y^s)^{-2}\\
s(x+y)^{s-1}(x^s+y^s)^{-1}-s(x+y)^sy^{s-1}(x^s+y^s)^{-2}
\end{array}\right]\\
&=s(x+y)^{s-1}(x^s+y^s)^{-2}\left[\begin{array}{l}
(x^s+y^s)-(x+y)x^{s-1}\\
(x^s+y^s)-(x+y)y^{s-1}
\end{array}\right]\\
&=s(x+y)^{s-1}(x^s+y^s)^{-2}(y^{s-1}-x^{s-1})\left[\begin{array}{ll}
y&
-x
\end{array}\right]^\bot.
\end{align*}
Notice that $\nabla f\rest{(x,\,y)}$ is orthogonal to $\left[\begin{array}{ll}x&y\end{array}\right]^\bot\sim(x,\,y)$.
Which means that the trajectories associated with the vector field $\nabla f$ are pieces of spheres. Moreover,
$\nabla f\rest{(x,\,y)}$ vanishes only at the straight lines $x=y$, $x=0$, and $y=0$, and we observe that $f$ is constant in those lines.
Now, 
for $s>1$ we have that $(y^{s-1}-x^{s-1})>0$ if, and only if, $y>x$, and it is straightforward to conclude that,
in the sphere containing a point $(a,\,b)$, with $a\ge0$, $b\ge0$, and $(a,\,b)\ne(0,\,0)$,
the function $f$ attains its minimum either at the line $x=0$ or at the line $y=0$;
and attains its maximum at the line $x=y$. Hence we can conclude that $1\le f(a,\,b)\le 2^{s-1}$.

Analogously, for $s<1$ we have that $(y^{s-1}-x^{s-1})>0$ if, and only if, $y<x$, which gives us $2^{s-1}\le f(a,\,b)\le 1$.
\qed

\subsection{A remark on the square root of a quadratic polynomial}\label{sApx:root_p2}

Let $p(x)$ be a polynomial, of degree two, in the variable $x\in\R^n$.
\begin{lemma}\label{L:triangleK}
If the Hessian matrix $\HH$ of $p$ is positive definite, then there is a constant $K\in\R$ such that for any
$x,y\in\R^n$ we have
\begin{equation}\label{poly}
 \norm{p(x+y)}{\R}^{\frac{1}{2}}\le \norm{p(x)}{\R}^{\frac{1}{2}}+\norm{p(y)}{\R}^{\frac{1}{2}}+K.
\end{equation}
\end{lemma}
\begin{proof}
We can see that the derivative $\ed_{x_0}p:\R^n\to\R$ can be rewritten as $x_0^\top\HH +G$,
for a suitable row matrix $G$, thus
there is (a unique) $\bar x\in\R^n$ such that $\bar x^\top\HH +G=0$.
Now we define the function $q(w)\coloneqq p(w+\bar x)-p(\bar x)$. From the Taylor formula we find that
\[
 q(w)=\textstyle\frac{1}{2} w^\top\HH w, 
\]
and we see that $q(w)=Q(w,\,w)$, where $Q(w,\,z)\coloneqq \frac{1}{2} w^\top\HH z$ is a scalar product in $\R^n$.

Now, from the identity $p(x+y)=q(x+y-\bar x)+p(\bar x)$, the inequality 
$(a+b)^{\frac{1}{2}}\le a^{\frac{1}{2}}+b^{\frac{1}{2}}$ (for $a,\,b\ge 0$, cf. Proposition~\ref{P:ab_s}), and the
triangle inequality $q(w+z)^{\frac{1}{2}}\le q(w)^{\frac{1}{2}}+q(z)^{\frac{1}{2}}$, it follows that
\begin{align*}
 \norm{p(x+y)}{\R}^{\frac{1}{2}}&\le q(x+y-\bar x)^{\frac{1}{2}}
 +\norm{p(\bar x)}{\R}^{\frac{1}{2}}\\
 &\le q(x-\bar x)^{\frac{1}{2}}+q(y-\bar x)^{\frac{1}{2}}
 +q(\bar x)^{\frac{1}{2}}
 +\norm{p(\bar x)}{\R}^{\frac{1}{2}}\\
 &\le \norm{p(x)}{\R}^{\frac{1}{2}}+\norm{p(y)}{\R}^{\frac{1}{2}}
 +\norm{p(0)}{\R}^{\frac{1}{2}}
 +4\norm{p(\bar x)}{\R}^{\frac{1}{2}}.
\end{align*}
Therefore, we may take
\begin{equation}\label{polyK}
K=\norm{p(0)}{\R}^{\frac{1}{2}}
 +4\norm{p(\bar x)}{\R}^{\frac{1}{2}} 
\end{equation}
in~\eqref{poly}.
\end{proof}

\subsection{On the curl operator in the Sphere}\label{sApx:curl}
Here we show that, in the case of the Sphere~$\Sp^2$, the definitions of the curl operators in~\cite{CaoRammahaTiti99} and
in~\cite{Rod-Thesis08} are equivalent up to a minus sign (cf. Remark~\ref{R:culr-def}).
Familiarity with basic tools from differential geometry is assumed;
we refer to~\cite{Cartan67,doCarmo94,Jost05,Trau84} (we follow the notations from~\cite[Chapter~5, Section~5.7]{Rod-Thesis08}).
Since the curl is a local operator it is enough to check those definitions on local charts. We consider the chart
\begin{align*}
 \Phi:\CCC&\to\BBB\\
 \Phi(w^1,\,w^2,\,w^3)&\mapsto(x^1,\,x^2,\,x^3)\coloneqq(1+w^3)(w^1,\,w^2,\,\Phi^0(w^1,\,w^2))
\end{align*}
with $\Phi^0(w^1,\,w^2)\coloneqq\left(1-(w^1)^2-(w^2)^2\right)^\frac{1}{2}$,
mapping the set $\CCC\coloneqq\{(w^1,\,w^2,\,w^3)\in\R^3\mid (w^1)^2+(w^2)^2<\frac{1}{2}\mbox{ and }w^3\le\frac{1}{2}\}$ onto
$\BBB\coloneqq\Phi(\CCC)\subset\{(x^1,\,x^2,\,x^3)\in\R^3\mid \frac{1}{2}<(x^1)^2+(x^2)^2+(x^3)^2<\frac{3}{2}\}$. Notice that we can
cover the entire Sphere~$\Sp^2$ with similar charts.

Let $\frac{\p}{\p w^i}$ be the vector field induced
in $\BBB$ by the new coordinate function $w^i$; we find
\begin{equation}\label{vf-pwi}
\begin{array}{l}
\frac{\p}{\p w^i}\rest {(w^1,\,w^2,\,w^3)}=(1+w^3)\left(\frac{\p}{\p
x^i} +\frac{\p\Phi_p^0}{\p
w^i}\rest {(w^1,\,w^2,\,w^3)}\frac{\p}{\p x^3}\right)\mbox{ for
}i=1,\,2\\
\frac{\p}{\p w^3}\rest {(w^1,\,w^2,\,w^3)}=w^1\frac{\p}{\p x^1}+w^2\frac{\p}{\p x^2}+\Phi^0(w^1,\,w^2)\frac{\p}{\p x^3}
=\nnn_{(w^1,\,w^2,\,\Phi_p^0(w^1,\,w^2))}
\end{array}
\end{equation}
where~$\nnn_{q}$ stands for the outward normal vector at the point
$q\in\Sp^2$ (cf.~\cite[Appendix]{Rod14}, recalling that
the outward normal vector at a point~$q=(q_1,\,q_2,\,q_3)\in\Sp^2$ is given by $\nnn_q=q_1\frac{\p}{\p x^1}+q_2\frac{\p}{\p x^2}+q_3\frac{\p}{\p x^3}\sim q$). Notice that $(w^1,\,w^2,\,\Phi^0(w^1,\,w^2)\in\Sp^2$. 

Reasoning, for example, as in~\cite[Appendix]{Rod14}, we can see the Euclidean set~$\BBB$ as the Riemannian manifold~$(\CCC,\,g)$ with the metric tensor
\begin{align*}
g&=\frac{1-(w^2)^2}{\Phi^0(w^1,\,w^2)^2}\ed w^1\otimes\ed w^1 +\frac{w^1w^2}{\Phi^0(w^1,\,w^2)^2}(\ed w^1\otimes\ed w^2+\ed w^2\otimes \ed w^1)\\
&\quad+\frac{1-(w^1)^2}{\Phi^0(w^1,\,w^2)^2}\ed w^2\otimes\ed w^2
 +\ed w^3\otimes\ed w^3
\end{align*}
and the Euclidean volume element in $\BBB$ may then be written as
$
\ed\CCC=\sqrt{\bar g}\,\ed w^1\wedge\ed w^2\wedge\ed w^3$, with $\bar g:=\frac{1}{\Phi^0(w^1,\,w^2)^2}.
$

Moreover the mapping $\Phi^0:\CCC_0\to\BBB_0$ maps the disc $\CCC_0\coloneqq\{(w^1,\,w^2)\in\R^2\mid (w^1)^2+(w^2)^2<\frac{1}{2}\}$ onto
$\BBB_0\coloneqq\BBB\cap\Sp^2$. Hence we can see the subset $\BBB_0$ with the metric inherited from~$\R^3$ as the
Riemannian manifold~$(\CCC_0,\,g_0)$ with the metric tensor
\begin{align*}
g_0&=\frac{1-(w^2)^2}{\Phi^0(w^1,\,w^2)^2}\ed w^1\otimes\ed w^1
+\frac{w^1w^2}{\Phi^0(w^1,\,w^2)^2}(\ed w^1\otimes\ed w^2+\ed w^2\otimes \ed w^1)\\
&\quad+\frac{1-(w^1)^2}{\Phi^0(w^1,\,w^2)^2}\ed w^2\otimes\ed w^2,\label{metric0}
\end{align*}
and volume (i.e., area) element
$\ed\CCC_0=\sqrt{\bar g_0}\,\ed w^1\wedge\ed w^2$ with $\bar g_0:=\frac{1}{\Phi^0(w^1,\,w^2)^2}=\bar g$.

\noindent
{\bf On the curl of a function.} In~\cite[Section~5.7]{Rod-Thesis08} the curl vector field $\curl f$ of a function~$f$ on
the Sphere is the vector field denoted~$\nabla^\perp f$ and defined as $\nabla^\perp f=(\ast\ed f)^\flat$,
where in coordinates~$(w^1,\,w^2)$, we denote by~$[g^{ij}]$ the inverse matrix~$[g_{ij}]^{-1}$ and
$(a_i\ed w^i)^\flat\coloneqq g^{ij}a_j\frac{\p}{\p w^i}$. We obtain
\begin{equation}\label{curl_inRod}
 \nabla^\perp f=\frac{1}{\sqrt{\bar g}}\left(\frac{\p f}{\p w^1}\frac{\p}{\p w^2}-\frac{\p f}{\p w^2}\frac{\p}{\p w^1}\right),
\end{equation}
while in~\cite[Definition~1]{CaoRammahaTiti99} it is denoted~$\mathop{\rm Curl}f$ and can be obtained as follows:
first we extend~$f$ to~$\BBB$;
then we consider the extension~$\tilde f\tilde\nnn$, where~$\tilde\nnn\coloneqq x_1\frac{\p}{\p x^1}+x_2\frac{\p}{\p x^2}
+x_3\frac{\p}{\p x^3}$ is an extension to $\BBB$ of the outward normal vector~$\nnn$ to $\Sp^2\supset\BBB_0$; finally we set
$\mathop{\rm Curl}f\coloneqq(\curl\tilde f\tilde\nnn)\rest{\BBB_0}$, where~$\curl=\nabla\times$ is the standard curl vector in~$\R^3$
(see also~\cite[Definition~1.1]{Ilyin91}, \cite[Definition~2.1]{Ilyin94}); Notice that we can
write~$\tilde f\tilde\nnn\rest{(x^1,\,x^2,\,x^3)}
=\frac{\tilde f}{\norm{\tilde\nnn}{\R^3}}\rest{(x^1,\,x^2,\,x^3)}\nnn\rest{(x^1,\,x^2,\,\Phi^0(x^1,\,x^2))}$. We find
\begin{align*}
 \curl\tilde f\tilde\nnn&=\curl\left(\tilde fx_1\frac{\p}{\p x^1}+\tilde fx_2\frac{\p}{\p x^2}+\tilde fx_3\frac{\p}{\p x^3}\right)\\
 &=\left(x_3\frac{\p\tilde f}{\p x^2}-x_2\frac{\p\tilde f}{\p x^3}\right)\frac{\p}{\p x^1}
 +\left(x_1\frac{\p\tilde f}{\p x^3}-x_3\frac{\p\tilde f}{\p x^1}\right)\frac{\p}{\p x^2}
 +\left(x_2\frac{\p\tilde f}{\p x^1}-x_1\frac{\p\tilde f}{\p x^2}\right)\frac{\p}{\p x^3}.
\end{align*}
and, from $x_3\rest{\BBB_0}=\Phi^0(x^1,\,x^2)$ we have
\begin{align}
 &\quad\mathop{\rm Curl} f\label{curl_inCRT}\\
 &=\left(\Phi^0\frac{\p\tilde f}{\p x^2}-x_2\frac{\p\tilde f}{\p x^3}\right)\frac{\p}{\p x^1}
 +\left(x_1\frac{\p\tilde f}{\p x^3}-\Phi^0\frac{\p\tilde f}{\p x^1}\right)\frac{\p}{\p x^2}
 +\left(x_2\frac{\p\tilde f}{\p x^1}-x_1\frac{\p\tilde f}{\p x^2}\right)\frac{\p}{\p x^3}.\notag
 \end{align}
On the other hand, from~\eqref{curl_inRod}, \eqref{vf-pwi}, the identity~$\sqrt{\bar g}=\frac{1}{\Phi^0}$, and from the fact that the vector
fields~$\frac{\p}{\p w^1}$ and $\frac{\p}{\p w^2}$ are tangent to~$\BBB_0$, we obtain
\begin{align}
 &\quad\nabla^\perp f=\left.\frac{1}{\sqrt{\bar g}}
 \left(\frac{\p\tilde f}{\p w^1}\frac{\p}{\p w^2}-\frac{\p\tilde f}{\p w^2}\frac{\p}{\p w^1}\right)\right|_{\{w^3=0\}}\\
 &=\left(x_2\frac{\p\tilde f}{\p x^3}-\Phi^0\frac{\p\tilde f}{\p x^2}\right)\frac{\p}{\p x^1}
 +\left(\Phi^0\frac{\p\tilde f}{\p x^1}-x_1\frac{\p\tilde f}{\p x^3}\right)\frac{\p}{\p x^2}
 +\left(x_1\frac{\p\tilde f}{\p x^2}-x_2\frac{\p\tilde f}{\p x^1}\right)\frac{\p}{\p x^3}.\notag
 \end{align}
That is, from~\eqref{curl_inCRT}, we have $\nabla^\perp f=-\mathop{\rm Curl} f$.

\noindent
{\bf On the curl of a vector field.} In~\cite[Appendix]{Rod14}, the $\curl$ of a vector field~$u\in T\CCC_0$,
in the manifold~$(\CCC_0,\,g_0)$, is the function defined and denoted as $\nabla^\perp\cdot u\coloneqq\ast\ed u^\sharp$, with
$\sharp=\flat^{-1}$, that is, in local coordinates $(V^i\frac{\p}{\p w_i})^\sharp=g_{ij}V^j\ed w^i$.
In~\cite[Definition~1]{CaoRammahaTiti99}, \cite[Section~2]{Ilyin94}, the $\curl$ of $u\in T\CCC_0$ is the function defined and denoted
as $\mathop{\rm Curl_\nnn} u\coloneqq((\curl\tilde u)\rest{\BBB_0},\,\nnn)_{\R^3}$,
where $\tilde u$ is an extension from~$\BBB_0$ to~$\BBB$ of~$u$. Now we can show that $\nabla^\perp\cdot u=\mathop{\rm Curl_\nnn} u$
up to an additive constant; we proceed as follows: first we notice that
the Laplacian in the two-dimensional manifold~$(\CCC_0,\,g_0)$ is defined by~$\Delta u=(-\ed\ast\ed\ast u^\sharp-\ast\ed\ast\ed u^\sharp)^\flat$
in~\cite[Appendix]{Rod14}, and given by~$(\ed\ast\ed\ast u^\sharp+\ast\ed\ast\ed u^\sharp)^\flat=-\Delta u=(\ed\ast\ed\ast u^\sharp)^\flat-\mathop{\rm Curl}\mathop{\rm Curl_\nnn} u$
in\cite[Section~2]{Ilyin94}. Necessarily, we have that $\ast\ed\ast\ed u^\sharp=-(\mathop{\rm Curl}\mathop{\rm Curl_\nnn} u)^\sharp$ for all~$u\in T\CCC_0$.
Since we already know that $\nabla^\perp f=-\mathop{\rm Curl} f$, it follows that
$\ast\ed(\ast\ed u^\sharp-\mathop{\rm Curl_\nnn} u)=0$, which implies that $\nabla^\perp\cdot u=\ast\ed u^\sharp=\mathop{\rm Curl_\nnn} u$ up to an additive constant.

Further, for a nonharmonic divergence free vector field we have that $\ast\ed u^\sharp=\mathop{\rm Curl_\nnn} u$. Indeed, if $u$ is
divergence free, that is if $-\ast\ed\ast u^\sharp=0$, then $\Delta u^\sharp=(\Delta u)^\sharp=-\ast\ed\ast\ed u^\sharp
=(\mathop{\rm Curl}\mathop{\rm Curl_\nnn} u)^\sharp
=-\ast\ed\mathop{\rm Curl_\nnn} u$. Now, for given constants~$c_1$ and~$c_2$, we have
$\Delta (c_2-c_1)u^\sharp=-\ast\ed(c_2\ast\ed u^\sharp-c_1\mathop{\rm Curl_\nnn} u)$; and if we write
$\ast\ed u^\sharp=z_1+\frac{\int_{\Sp^2}\ast\ed u^\sharp\ed\Sp^2}{\int_{\Sp^2}1\ed\Sp^2 u}$ and
$\mathop{\rm Curl_\nnn} u=z_2+\frac{\int_{\Sp^2}\mathop{\rm Curl_\nnn} u\ed\Sp^2}{\int_{\Sp^2}1\ed\Sp^2 u}$ we have that
$z_j$ is zero averaged, $\int_{\Sp^2}z_j\ed\Sp^2=0$, for $j\in\{1,\,2\}$. Choosing
$c_1=\frac{\int_{\Sp^2}\ast\ed u^\sharp\ed\Sp^2}{\int_{\Sp^2}1\ed\Sp^2 u}$
and $c_2=\frac{\int_{\Sp^2}\mathop{\rm Curl_\nnn} u\ed\Sp^2}{\int_{\Sp^2}1\ed\Sp^2 u}$, we obtain
$(c_2-c_1)\Delta u^\sharp=-\ast\ed(c_2z_1-c_1z_2)$ and $c_2z_1-c_1z_2=c_2\ast\ed u^\sharp-c_1\mathop{\rm Curl_\nnn} u$. Since $c_2z_1-c_1z_2$
is constant and zero averaged, necessarily $c_2z_1-c_1z_2=0$, which implies $c_1=c_2$, because $u$ is nonharmonic.

If $c_1=c_2\ne0$, we have that $\nabla^\perp\cdot u=\ast\ed u^\sharp=z_1+c_1=z_2+c_2=\mathop{\rm Curl_\nnn} u$; if $c_1=c_2=0$ we have that
$\ast\ed u^\sharp-\mathop{\rm Curl_\nnn} u=z_1-z_2$ is constant and zero averaged, so again $\nabla^\perp\cdot u=\ast\ed u^\sharp=\mathop{\rm Curl_\nnn} u$.


Furthermore, notice that by our choice of the space~$H$ in Section~\ref{S:examples}, harmonic vector fields are
orthogonal to the space $H$. Indeed under Lions boundary conditions we have $A=\Delta$ and the eigenfunctions
$\{W_k\mid k\in\N_0\}$ of the Stokes operator $A$ with positive eigenvalues form a basis in~$H$,
and any divergence free harmonic vector field $\WW$ satisfies~$(\WW,\,W_k)_H=(\WW,\,\lambda_k^{-1}\Delta W_k)_H
=\lambda_k^{-1}(\WW,\,\Delta W_k)_H
=\lambda_k^{-1}(\Delta\WW,\,W_k)_H=0$. Therefore we can conclude that
\[
 \nabla^\perp\cdot u=\mathop{\rm Curl_\nnn} u\in H^{-1}(\Sp^2,\,\R)\coloneqq H^1(\Sp^2,\,\R)'\quad\mbox{for all }u\in H
\]
where $(\nabla^\perp\cdot u,\,v)_{H^{-1}(\Sp^2,\,\R),\,H^1(\Sp^2,\,\R)}\coloneqq-(u,\,\nabla^\perp v)_H$. Notice that for
smoother data $u\in H^1(\Sp^2,\,T\Sp^2)$ and $v\in H^1(\Sp^2,\,\R)$, we can write
\begin{align*}
&\quad(\nabla^\perp\cdot u,\,v)_{H^{-1}(\Sp^2,\,\R),\,H^1(\Sp^2,\,\R)}=
(\nabla^\perp\cdot u,\,v)_{L^2(\Sp^2,\,\R)}=\int_{\Sp^2}(\ast\ed u^\sharp)v\,\ed\Sp^2
=\int_{\Sp^2}\ast (v\ed u^\sharp)\,\ed\Sp^2\\
&=\int_{\Sp^2}\ed (v u^\sharp)-\int_{\Sp^2}\ed v\wedge u^\sharp
=0-\int_{\Sp^2}\iota_u(*\ed v)\,\ed\Sp^2
=-\int_{\Sp^2}g(u,\,(*\ed v)^\flat)\,\ed\Sp^2\\
&=-(u,\,\nabla^\perp v)_{L^2(\Sp^2,\,T\Sp^2)}.
\end{align*}






\begin{thebibliography}{FMRT01}

\bibitem[Arn66]{Arnold66}
V.~I. Arnold.
\newblock Sur la g\'eom\'etrie diff\'erentielle des groupes de {L}ie de
  dimension infinie et ses applications \`a l'hydrodynamique des fluides
  parfaits.
\newblock {\em Ann. Inst. Fourier}, 16(1):319--361, 1966.
\newblock \href {http://dx.doi.org/10.5802/aif.233}
  {\path{doi:10.5802/aif.233}}.

\bibitem[AS05]{AgraSary05}
A.~A. Agrachev and A.~V. Sarychev.
\newblock {N}avier--{S}tokes equations: Controllability by means of low modes
  forcing.
\newblock {\em J. Math. Fluid Mech.}, 7(1):108--152, 2005.
\newblock \href {http://dx.doi.org/10.1007/s00021-004-0110-1}
  {\path{doi:10.1007/s00021-004-0110-1}}.

\bibitem[AS11]{AmroucheSeloula11}
C.~Amrouche and N.~Seloula.
\newblock On the {S}tokes equations with the {N}avier-type boundary conditions.
\newblock {\em Differ. Equ. Appl.}, 3(4):581--607, 2011.
\newblock \href {http://dx.doi.org/10.7153/dea-03-36}
  {\path{doi:10.7153/dea-03-36}}.

\bibitem[Car67]{Cartan67}
H.~Cartan.
\newblock {\em Formes Diff\'erentielles}.
\newblock Collection M\'ethodes. Hermann Paris, 1967.
\newblock URL: \url{http://store.doverpublications.com/0486450104.html}.

\bibitem[CCG10]{ChemetovCiprianoGavrilyuk10}
N.~V. Chemetov, F.~Cipriano, and S.~Gavrilyuk.
\newblock Shallow water model for lakes with friction and penetration.
\newblock {\em Math. Meth. Appl. Sci.}, 33(6):687--703, 2010.
\newblock \href {http://dx.doi.org/10.1002/mma.1185}
  {\path{doi:10.1002/mma.1185}}.

\bibitem[CF96]{CoronFur96}
J.-M. Coron and A.~V. Fursikov.
\newblock Global exact controllability of the {N}avier--{S}tokes equations on a
  manifold without boundary.
\newblock {\em Russian J. Math. Phys.}, 4(4):429--448, 1996.
\newblock URL:
  \url{https://www.ljll.math.upmc.fr/~coron/Documents/1996rjmp.pdf}.

\bibitem[CRT99]{CaoRammahaTiti99}
C.~Cao, M.~A. Rammaha, and E.~S. Titi.
\newblock The {N}avier--{S}tokes equations on the rotating 2-{D} sphere:
  {G}evrey regularity and asymptotic degrees of freedom.
\newblock {\em Z. angew. Math. Thys.}, 50(3):341--360, 1999.
\newblock \href {http://dx.doi.org/10.1007/PL00001493}
  {\path{doi:10.1007/PL00001493}}.

\bibitem[dC94]{doCarmo94}
M.~P. do~Carmo.
\newblock {\em Differential Forms and Applications}.
\newblock Universitext. Springer, 1994.
\newblock \href {http://dx.doi.org/10.1007/978-3-642-57951-6}
  {\path{doi:10.1007/978-3-642-57951-6}}.

\bibitem[DD12]{DemengelDem12}
F.~Demengel and G.~Demengel.
\newblock {\em Functional Spaces for the Theory of Elliptic Partial
  Differential Equations}.
\newblock Universitext. Springer, 2012.
\newblock \href {http://dx.doi.org/10.1007/978-1-4471-2807-6}
  {\path{doi:10.1007/978-1-4471-2807-6}}.

\bibitem[Edm96]{Edmonds96}
A.~R. Edmonds.
\newblock {\em Angular Momentum in Quantum Mechanics}.
\newblock Landmarks in Mathematics and Physics. Princeton University Press, 4th
  printing edition, 1996.
\newblock URL: \url{http://press.princeton.edu/titles/478.html}.

\bibitem[FF05]{FenglerFreeden05}
M.~J. Fengler and W.~Freeden.
\newblock A nonlinear {G}alerkin scheme involving vector and tensor spherical
  harmonics for solving the incompressible {N}avier--{S}tokes equation on the
  sphere.
\newblock {\em SIAM J. Sci. Comput.}, 27(3):967--994, 2005.
\newblock \href {http://dx.doi.org/10.1137/040612567}
  {\path{doi:10.1137/040612567}}.

\bibitem[FMRT01]{FoiManRosaTem01}
C.~Foias, O.~Manley, R.~Rosa, and R.~Temam.
\newblock {\em {N}avier--{S}tokes Equations and Turbulence}.
\newblock Encyclopedia of Mathematics and its Applications. Cambridge
  University Press, 2001.
\newblock \href {http://dx.doi.org/10.1017/CBO9780511546754}
  {\path{doi:10.1017/CBO9780511546754}}.

\bibitem[F{\v N}05]{FeireislNecasova11}
E.~Feireisl and {\v S}.~{\v N}ecasov{\'a}.
\newblock The effective boundary conditions for vector fields on domains with
  rough boundaries: Applications to fluid mechanics.
\newblock {\em Appl. Math.}, 56(1):39--49, 2005.
\newblock URL: \url{http://dml.cz/dmlcz/141405}.

\bibitem[FT89]{FoiasTemam89}
C.~Foia{\c{s}} and R.~Temam.
\newblock {G}evrey class regularity for the solutions of the {N}avier--{S}tokes
  equations.
\newblock {\em J. Funct. Anal.}, 87(2):359--369, 1989.
\newblock \href {http://dx.doi.org/10.1016/0022-1236(89)90015-3}
  {\path{doi:10.1016/0022-1236(89)90015-3}}.

\bibitem[Ily91]{Ilyin91}
A.~A. Ilyin.
\newblock The {N}avier--{S}tokes and {E}uler equations on two-dimensional
  closed manifolds.
\newblock {\em Math. USSR-Sb.}, 69(2):559--579, 1991.
\newblock \href {http://dx.doi.org/10.1070/SM1991v069n02ABEH002116}
  {\path{doi:10.1070/SM1991v069n02ABEH002116}}.

\bibitem[Ily94]{Ilyin94}
A.~A. Ilyin.
\newblock Partly dissipative semigroups generated by the {N}avier--{S}tokes
  system on two-dimensional manifolds, and their attractors.
\newblock {\em Russian Acad. Sci. Sb. Math.}, 78(1):47--76, 1994.
\newblock \href {http://dx.doi.org/10.1070/SM1994v078n01ABEH003458}
  {\path{doi:10.1070/SM1994v078n01ABEH003458}}.

\bibitem[Ily09]{Ilyin09}
A.~A. Ilyin.
\newblock On the spectrum of the {S}tokes operator.
\newblock {\em Funct. Anal. Appl.}, 43(4):254--263, 2009.
\newblock \href {http://dx.doi.org/10.1007/s10688-009-0034-x}
  {\path{doi:10.1007/s10688-009-0034-x}}.

\bibitem[IP06]{IftimiePlanas06}
D.~Iftimie and G.~Planas.
\newblock Inviscid limits for the {N}avier--{S}tokes equations with {N}avier
  friction boundary conditions.
\newblock {\em Nonlinearity}, 19(4):899--918, 2006.
\newblock \href {http://dx.doi.org/10.1088/0951-7715/19/4/007}
  {\path{doi:10.1088/0951-7715/19/4/007}}.

\bibitem[IT06]{IlyinTiti06}
A.~A. Ilyin and E.~S. Titi.
\newblock Sharp estimates for the number of degrees of freedom for the
  damped-driven {2-D} {N}avier--{S}tokes equations.
\newblock {\em J. Nonlinear Sci.}, 16(3):233--253, 2006.
\newblock \href {http://dx.doi.org/10.1007/s00332-005-0720-7}
  {\path{doi:10.1007/s00332-005-0720-7}}.

\bibitem[JM01]{JagerMikelic01}
W.~J{\"a}ger and A.~Mikeli{\'c}.
\newblock On the roughness-induced effective boundary conditions for an
  incompressible viscous flow.
\newblock {\em J. Differential Equations}, 170(1):96--122, 2001.
\newblock \href {http://dx.doi.org/10.1006/jdeq.2000.3814}
  {\path{doi:10.1006/jdeq.2000.3814}}.

\bibitem[Jos05]{Jost05}
J.~Jost.
\newblock {\em Riemannian Geometry and Geometric Analysis}.
\newblock Univesitext. Springer, 4th edition, 2005.
\newblock \href {http://dx.doi.org/10.1007/978-3-642-21298-7}
  {\path{doi:10.1007/978-3-642-21298-7}}.

\bibitem[Kel06]{Kelliher06}
J.~P. Kelliher.
\newblock {N}avier--{S}tokes equations with {N}avier boundary conditions for a
  bounded domain in the plane.
\newblock {\em SIAM J. Math. Anal.}, 38(1):210--232, 2006.
\newblock \href {http://dx.doi.org/10.1137/040612336}
  {\path{doi:10.1137/040612336}}.

\bibitem[Lio69]{Lions69}
J.-L. Lions.
\newblock {\em Quelques M\'ethodes de R\'esolution des Probl\`emes aux Limites
  Non Lin\'eaires}.
\newblock Dunod et Gauthier--Villars, 1969.

\bibitem[Liu92]{Liu92}
X.~Liu.
\newblock A note on {G}evrey class regularity for the solutions of the
  {N}avier--{S}tokes equations.
\newblock {\em J. Math. Anal. Appl.}, 167(2):588--595, 1992.
\newblock \href {http://dx.doi.org/10.1016/0022-247X(92)90226-4}
  {\path{doi:10.1016/0022-247X(92)90226-4}}.

\bibitem[LM72]{LioMag72-I}
J.-L. Lions and E.~Magenes.
\newblock {\em Non-Homogeneous Boundary Value Problems and Applications},
  volume~I of {\em Die Grundlehren der Mathematischen Wissenschaften in
  Einzeldarstellungen}.
\newblock Springer-Verlag, 1972.
\newblock \href {http://dx.doi.org/10.1007/978-3-642-65161-8}
  {\path{doi:10.1007/978-3-642-65161-8}}.

\bibitem[LY83]{LiYau83}
P.~Li and S.-T. Yau.
\newblock On the {S}chr{\"o}dinger equation and the eigenvalue problem.
\newblock {\em Commun. Math. Phys.}, 88(3):309--318, 1983.
\newblock \href {http://dx.doi.org/10.1007/BF01213210}
  {\path{doi:10.1007/BF01213210}}.

\bibitem[Pri94]{Priebe94}
V.~Priebe.
\newblock Solvability of the {N}avier--{S}tokes equations on manifolds with
  boundary.
\newblock {\em Manuscripta Math.}, 83(1):145--159, 1994.
\newblock \href {http://dx.doi.org/10.1007/BF02567605}
  {\path{doi:10.1007/BF02567605}}.

\bibitem[Rod06]{Rod06}
S.~S. Rodrigues.
\newblock {N}avier--{S}tokes equation on the {R}ectangle: Controllability by
  means of low modes forcing.
\newblock {\em J. Dyn. Control Syst.}, 12(4):517--562, 2006.
\newblock \href {http://dx.doi.org/10.1007/s10883-006-0004-z}
  {\path{doi:10.1007/s10883-006-0004-z}}.

\bibitem[Rod07]{Rod-wmctf07}
S.~S. Rodrigues.
\newblock Controllability of nonlinear pdes on compact {R}iemannian manifolds.
\newblock In {\em Proceedings WMCTF'07; Lisbon, Portugal}, pages 462--493,
  April 2007.
\newblock URL: \url{http://people.ricam.oeaw.ac.at/s.rodrigues/}.

\bibitem[Rod08]{Rod-Thesis08}
S.~S. Rodrigues.
\newblock {\em Methods of Geometric Control Theory in Problems of Mathematical
  Physics}.
\newblock PhD Thesis. Universidade de Aveiro, Portugal, 2008.
\newblock URL: \url{http://hdl.handle.net/10773/2931}.

\bibitem[Rod14]{Rod14}
S.~S. Rodrigues.
\newblock Local exact boundary controllability of {3D} {N}avier--{S}tokes
  equations.
\newblock {\em Nonlinear Anal.}, 95:175--190, 2014.
\newblock \href {http://dx.doi.org/10.1016/j.na.2013.09.003}
  {\path{doi:10.1016/j.na.2013.09.003}}.

\bibitem[RY04]{RaschYu04}
J.~Rasch and A.~C.~H. Yu.
\newblock Efficient storage scheme for precalculated {W}igner {3j}, {6j} and
  {G}aunt coefficients.
\newblock {\em SIAM J. Sci. Comput.}, 25(4):1416--1428, 2004.
\newblock \href {http://dx.doi.org/10.1137/S1064827503422932}
  {\path{doi:10.1137/S1064827503422932}}.

\bibitem[Tay97]{Taylor97}
M.~E. Taylor.
\newblock {\em Partial Differential Equations {I} - {B}asic Theory}.
\newblock Number 115 in Applied Mathematical Sciences. Springer, 1997.
\newblock (corrected 2nd printing).
\newblock \href {http://dx.doi.org/10.1007/978-1-4419-7055-8}
  {\path{doi:10.1007/978-1-4419-7055-8}}.

\bibitem[Tem95]{Temam95}
R.~Temam.
\newblock {\em {N}avier--{S}tokes Equations and Nonlinear Functional Analysis}.
\newblock Number~66 in CBMS-NSF Regional Conference Series in Applied
  Mathematics. SIAM, 2nd edition, 1995.
\newblock \href {http://dx.doi.org/10.1137/1.9781611970050}
  {\path{doi:10.1137/1.9781611970050}}.

\bibitem[Tem97]{Temam97}
R.~Temam.
\newblock {\em Infinite-dimensional Dynamical Systems in {M}echanics and
  {P}hysics}.
\newblock Number~68 in Applied Mathematical Sciences. Springer, 2nd edition,
  1997.
\newblock \href {http://dx.doi.org/10.1007/978-1-4612-0645-3}
  {\path{doi:10.1007/978-1-4612-0645-3}}.

\bibitem[Tem01]{Temam01}
R.~Temam.
\newblock {\em {N}avier--{S}tokes Equations: Theory and Numerical Analysis}.
\newblock AMS Chelsea Publishing, {reprint of the 1984} edition, 2001.
\newblock URL: \url{http://www.ams.org/bookstore-getitem/item=CHEL-343-H}.

\bibitem[Tra84]{Trau84}
A.~Trautman.
\newblock {\em Differential Geometry for Physicists, {S}tony {B}rook
  {L}ectures}.
\newblock Monographs and Textbooks in Physical Science. Bibliopolis, 1984.
\newblock URL: \url{http://www.bibliopolis.it/index2.htm}.

\bibitem[WW69]{WhittakerWatson69}
E.~T. Whittaker and G.~N. Watson.
\newblock {\em A course of Modern Analysis}.
\newblock Cambridge University Press, reprinted 4th edition, 1969.
\newblock URL: \url{http://www.cambridge.org/}.

\bibitem[WXZ12]{WangXinZang12}
L.~Wang, Z.~Xin, and A.~Zang.
\newblock Vanishing viscous limits for 3d navier–stokes equations with a
  navier-slip boundary condition.
\newblock {\em J. Math. Fluid Mech.}, 14(4):791--825, 2012.
\newblock \href {http://dx.doi.org/10.1007/s00021-012-0103-4}
  {\path{doi:10.1007/s00021-012-0103-4}}.

\bibitem[XX07]{XiaoXin07}
Y.~Xiao and Z.~Xin.
\newblock On the vanishing viscosity limit for the {3D} {N}avier--{S}tokes
  equations with a slip boundary condition.
\newblock {\em Comm. Pure Appl. Math.}, 60(7):1027--1055, 2007.
\newblock \href {http://dx.doi.org/10.1002/cpa.20187}
  {\path{doi:10.1002/cpa.20187}}.

\end{thebibliography}

\end{document}